\documentclass[a4paper,12pt]{article}
\usepackage{enumerate}
\usepackage{amsmath,amssymb,amsthm, mathrsfs}
\makeatletter
\renewcommand*\env@matrix[1][*\c@MaxMatrixCols c]{%
  \hskip -\arraycolsep
  \let\@ifnextchar\new@ifnextchar
  \array{#1}}
\makeatother
\usepackage{latexsym}
\usepackage{graphics}
\usepackage{hyperref}
\usepackage[bf, small]{titlesec}
\usepackage{authblk} %%% new line between authors
\usepackage{soul}
\usepackage{dhucs}
\usepackage{lineno}
\usepackage{tikz}

\definecolor{LemonChiffon}{rgb}{100, 98, 80}
\definecolor{myblue}{rgb}{0,0.4,0.8}
\definecolor{orange}{rgb}{1, 0.4, 0}
\definecolor{mygreen}{rgb}{0, 0.8, 0.2}
\definecolor{myred}{rgb}{204, 0, 0}
\definecolor{violet}{RGB}{0.4,0.2,1}
\definecolor{brown}{rgb}{0.6, 0.4, 0}

\theoremstyle{plain}
\newtheorem{Thm}{Theorem}[section]
\newtheorem{Lem}[Thm]{Lemma}
\newtheorem{Prop}[Thm]{Proposition}
\newtheorem{Cor}[Thm]{Corollary}

\newtheorem*{claim1}{Claim A}
\newtheorem*{claim2}{Claim B}
\newtheorem*{claim3}{Claim C}

\theoremstyle{definition}

\usepackage{standalone}
\usepackage{pgfpages,tikz,pgfkeys,pgfplots}
\usetikzlibrary{arrows,positioning,matrix,fit,backgrounds,shapes,shapes.geometric, intersections}

\usetikzlibrary{shapes.callouts,decorations.pathmorphing} %%% 팝업용
\usepackage{calc} % calculate angles to evenly space vertices in circular arrangements.
\usetikzlibrary{decorations.markings} %%put arrowheads in the middle of directed edges
\tikzstyle{vertex}=[circle, draw, inner sep=0pt, minimum size=6pt] % style

\usetikzlibrary{arrows,matrix} %%% Hesse Diagram

\title{Extensions of results on phylogeny graphs of degree bounded digraphs}
\author[]{Myungho Choi\thanks{Corresponding author
\\E-mail addresses: nums8080@snu.ac.kr (M.Choi), srkim@snu.ac.kr (S.-R.Kim)
}
}
\author[]{Suh-Ryung Kim
}

\affil[]{Department of Mathematics Education,
Seoul National University, Seoul 08826, Republic of Korea}

\begin{document}
\maketitle
\begin{abstract}
An acyclic digraph in which every vertex has indegree at most $i$ and outdegree at most $j$ is called an $(i,j)$ digraph for some positive integers 
$i$ and $j$.
The phylogeny graph of a digraph $D$ has $V(D)$ as the vertex set and an edge $uv$ if and only if one of the following is true: $(u,v) \in A(D)$; $(v,u) \in A(D)$;  $(u,w) \in A(D)$ and $(v,w) \in A(D)$ for some $w \in V(D)$.
 A graph $G$ is a phylogeny graph (resp.\ an $(i,j)$ phylogeny graph) if there is an acyclic digraph $D$ (resp.\ an $(i,j)$ digraph $D$) such that the phylogeny graph of $D$ is isomorphic to $G$.
Lee~{\em et al.} (2017) and Eoh and Kim (2021) studied the $(2,2)$ phylogeny graphs, $(1,j)$ phylogeny graphs, $(i,1)$ phylogeny graphs, and $(2,j)$ phylogeny graphs.
Their work was motivated by problems related to evidence propagation in a Bayesian network for which it is useful to know which acyclic digraphs have chordal moral graphs (phylogeny graphs are called moral graphs in Bayesian network theory).
In this paper,
we extend their work by 
giving necessary conditions of chordal $(i,2)$ phylogeny graphs.
We go further to give necessary conditions of $(i,j)$ phylogeny graphs by listing forbidden induced subgraphs.
\end{abstract}


\section{Introduction}
Throughout this paper, we deal with simple graphs and simple acyclic digraphs.

Given an acyclic digraph $D$, the competition graph of $D$, denoted by $C(D)$, is the simple graph having vertex set $V(D)$ and
edge set \[\{uv \mid (u, w), (v, w) \in A(D) \text{ for some $w \in V (D)$}\}.\] 
Since Cohen \cite{cohen} introduced the notion of competition graphs in
the study of predator–prey concepts in ecological food webs, multiple variants of competition graphs have been introduced
and studied.
For recent work related to competition graphs, see \cite{factor20111,KA,LiChang,ZR}. 

The notion of phylogeny graphs was introduced by Roberts and Sheng \cite{roberts1997phylogeny} as a variant of competition graphs. (See also 
\cite{hartke2005elimination,park2013phylogeny,Roberts98extremalphylogeny,roberts1998phylogeny,roberts2000phylogeny,zhao2006note} for more on phylogeny graphs.) 
Given an acyclic digraph $D$, the {\it underlying graph} of $D$, denoted by $U(D)$, is the simple graph with vertex set $V(D)$ and edge set $\{uv\mid (u,v) \in A(D) \text{ or } (v,u) \in A(D)\}$.
The {\it phylogeny graph} of an acyclic digraph $D$, denoted by $P(D)$, is the graph with the vertex set $V(D)$ and edge set $E(U(D))\cup E(C(D))$.
An edge is called a {\it cared edge} in $P(D)$ if the edge belongs to $E(C(D))$ but not to $E(U(D))$. 
For a cared edge $xy$ in $P(D)$, there
is a common out-neighbor $v$ of $x$ and $y$ in $D$ and it is said that $xy$ is {\it taken care} of by $v$ or that
$v$ {\it takes care} of $xy$. A vertex in $D$ is called a {\it caring vertex} if an edge of $P(D)$ is taken
care of by the vertex.

Given a cycle of a graph, a {\it chord} of the cycle is an edge between nonconsecutive vertices.
Given a graph, we call an induced subgraph that is a chordless a cycle of length at least 4 a {\it hole}.
 A graph is said to be {\it chordal} if it does not contain a hole.

“Moral graphs”, having arisen from studying Bayesian networks, are the same as phylogeny graphs. One of the best-known problems, in the context of Bayesian networks, is related to the propagation of evidence. It consists of the assignment of probabilities to the values of the rest of the variables, once the values of some variables are known. Cooper \cite{cooper1990computational} showed that this problem is NP-hard. Noteworthy algorithms for this problem are given by Pearl \cite{pearl1986fusion}, Shachter \cite{shachter1988probabilistic}, and Lauritzen and Spiegelhalter \cite{lauritzen1988local}. Those algorithms include a step of triangulating a moral graph, that is, adding proper edges to a moral graph to form a chordal graph.

 Stief \cite{steif1982frame} showed that there is a forbidden subgraph characterization of acyclic digraphs whose competition graphs are interval.
In that respect, Hefner~{\em et al}. \cite{hefner1991j} placed restrictions on the indegree and the outdegree of vertices of acyclic digraphs to obtain the list of forbidden subdigraphs for acyclic digraphs whose competition graphs are interval.
They called an acyclic digraph with indegree at most $i$ and outdegree at most $j$ an $(i,j)$ {\it digraph} for positive integers $i$ and $j$.

 A graph $G$ is an $(i,j)$ {\it phylogeny graph} if there is an $(i,j)$ digraph $D$ such that $P(D)$ is isomorphic to $G$.
Throughout this paper, we assume that variables $i$ and $j$ belong to the set of positive integers unless otherwise stated.

We present two main theorems in this paper.
One of them gives a necessary condition for an $(i,2)$ digraph having a chordal phylogeny graph as follows:
 \begin{Thm}\label{thm:main}
Let $H$ be a hole with length $l$ in the underlying graph of an $(i,2)$ digraph $D$.
If $l \geq 3i+1$,
then the subgraph of the phylogeny graph of $D$ induced by $V(H)$ has a hole.
Furthermore, for all $i\geq 2$, there is an $(i,2)$ digraph $D$ with hole $H$ of length $l=3i$ in $U(D)$
such that the subgraph of $P(D)$ induced by $V(H)$ does not have a hole.
\end{Thm}
Theorem~\ref{thm:main} extends the following theorem given by Lee~{\em et al}.~\cite{lee2017phylogeny}.
\begin{Thm}[\cite{lee2017phylogeny}] \label{thm:(2,2)-hole} Let $D$ be a $(2,2)$ digraph.
If the underlying graph of $D$ contains a hole $H$ of length at least $7$,
then the subgraph of the phylogeny graph of $D$ induced by $V(H)$ has a hole.
\end{Thm}

Lee~{\em et al}.~\cite{lee2017phylogeny} gave a forbidden induced subgraph for $(2,2)$ phylogeny graphs as follows:
\begin{Thm} [\cite{lee2017phylogeny}]
\label{thm:(2,2)-K_5-free}
For any $(2,2)$ digraph $D$, the phylogeny graph of $D$ is $K_5$-free. 
\end{Thm}
The {\it join} of two graphs $G$ and $H$, denoted by $G \vee H$, is the graph formed from disjoint copies of $G$ and $H$ by joining each vertex of $G$ to each vertex of $H$ by an edge. 

The following theorem shows that $P_{7} \vee I_1$, $C_{7} \vee I_1$, $K_{1,4}$, and $K_{3,3}$ are also forbidden induced subgraphs of $(2,2)$ phylogeny graphs other than $K_5$.

\begin{Thm} \label{thm:final_forbidden}
For $i,j \geq 2$, if a graph $G$ contains \[K_{1,j+2},\quad K_{j+1,j+1},\quad  P_{2j+3}\vee I_1,\quad  C_{2j+3}\vee I_1,\quad  \mbox{or} \quad K_{ij+1}\] as an induced subgraph, then $G$ is not the phylogeny graph of an (i,j) digraph. Furthermore, if $i\geq 4$ and $j=2$ and a graph $G$ contains $K_{\lfloor \frac{3i}{2}\rfloor +2}$ as an induced subgraph, then $G$ is not the phylogeny graph of an $(i,j)$ digraph.
\end{Thm}
We denote the set of out-neighbors and the set of in-neighbors of a vertex $v$ in a digraph $D$ by $N_D^+(v)$ and $N_D^-(v)$, respectively.
In addition, we denote the set of neighbors of a vertex $v$ in a graph $G$ by $N_G(v)$.
When no confusion is likely to occur,
 we omit $D$ or $G$ to just write $N^+(v)$, $N^-(v)$, and $N(v)$.

\section{A necessary condition for an $(i,2)$ phylogeny graph to be chordal}\label{sec:chordal}

Let $D$ be an acyclic digraph.
Suppose that the underlying graph $U(D)$ of $D$ has a hole $H=v_1v_2 \cdots v_lv_1$ of length $l$ for some $l \geq 5$.
Let $G$ be the subgraph of the phylogeny graph $P(D)$ induced by $V(H)$ and $D_H$ be the subdigraph of $D$ induced by $V(H)$.
Then, since each vertex has degree $2$ in $U(D_H)$, each of $v_1,v_2,\ldots,v_l$ has (i) exactly two in-neighbors, or (ii) exactly one out-neighbor and exactly one in neighbor, or (iii) exactly two out-neighbors in $D_H$.
Since $D$ is acyclic, $D_H$ is acyclic and so there exists a vertex in $D_H$ of indegree $2$.
Let $v_{l_1},v_{l_2},\ldots,v_{l_k}$ be the vertices in $V(H)$ having two in-neighbors in $D_H$ for an integer $k \geq 1$.
We denote the set $\{v_{l_1},v_{l_2},\ldots,v_{l_k}\}$ by $\Gamma_H$.
Then any two vertices in $\Gamma_H$ do not lie consecutively on $H$ and so $1\leq k \leq \lfloor \frac{l}{2}\rfloor$.
Therefore we obtain the cycle $C$ of length $l-k$ in $P(D)$ by deleting $v_{l_1},v_{l_2},\ldots,v_{l_k}$ from $G$ satisfying the property that each edge of $C$ either is taken care of some vertex in $\Gamma_H$ or lies on $H$.
We call such a cycle the {\it cycle obtained from $H$ by $\Gamma_H$}.
When no confusion is likely to occur, we omit $\Gamma_H$ in the cycle obtained from $H$ by $\Gamma_H$ to just write the cycle obtained from $H$.
We note that the length of $C$ is at least $l- \lfloor \frac{l}{2}\rfloor$ and at most $l-1$.
Moreover, by (ii) and (iii), each vertex on $C$ has at least one out-neighbor in $V(H)$.
Since the average outdegree of the vertices in $V(H)$ is $1$, the number of vertices on $C$ having two out-neighbors in $V(H)$ is equal to $|\Gamma_H|$.
In addition, there is no arc in $V(H)$ between nonconsecutive vertices on $C$, which implies that each chord of $C$ is a cared edge by a vertex in $V(D) - V(H)$.
Hence we immediately have the following lemma.
\begin{Lem} \label{lem:char-cycle}
Let $H$ be a hole with length $l \geq 5$ in the underlying graph of an acyclic digraph $D$ and $C$ be the cycle obtained from $H$ by $\Gamma_H$.
Then the following are true:
\begin{enumerate}[(1)]
\item the length of $C$ is at least $l-\lfloor \frac{l}{2}\rfloor$ and at most $l-1$;
\item each vertex on $C$ has at least one out-neighbor in $V(H)$;
\item the number of vertices on $C$ having two out-neighbors in $V(H)$ is equal to $|\Gamma_H|$;
\item for each chord $uv$ of $C$ in $P(D)$, $uv$ is a cared edge and each vertex taking care of $uv$ belongs to $V(D) - V(H)$.
\end{enumerate}
\end{Lem}
We obtain some useful characteristics
 on the cycle obtained from a hole in the underlying graph of an $(i,2)$ digraph as follows.
\begin{Prop} \label{prop:chord_property}
Let $H$ be a hole with length $l \geq 5$ in the underlying graph of an $(i,2)$ digraph $D$ and $C$ be the cycle obtained from $H$.
Suppose that $C$ has a chord $uv$ in $P(D)$.
Then the following are true:
\begin{enumerate}[(1)]
\item there exists exactly one vertex $w$ taking care of $uv$ in $D$;
\item $w$ is the only out-neighbor in $V(D)- V(H)$ of each of $u$ and $v$;
     \item for the subgraph induced by the chords of $C$ in $P(D)$, if $T$ is its component containing $uv$, then $w$ is the common out-neighbor in $D$ of the
     vertices in $T$ and $V(T)$ forms a clique in $P(D)$.
\end{enumerate}
\end{Prop}
\begin{proof}
Since $uv$ is a chord of $C$ in $P(D)$,
$uv$ is a cared edge
and there exists a vertex $w$ taking care of $uv$, which belongs to $V(D)- V(H)$ by Lemma~\ref{lem:char-cycle}(4).
Then $w$ is a common out-neighbor of $u$ and $v$.
Since $D$ is an $(i,2)$ digraph, each of $u$ and $v$ has outdegree at most $2$.
Then, by Lemma~\ref{lem:char-cycle}(2),
each of $u$ and $v$
has at most one out-neighbor in $V(D)- V(H)$.
Therefore
$w$ is the only out-neighbor of each of $u$ and $v$ in $V(D)- V(H)$ and so
$w$ is the only vertex taking care of $uv$ in $D$.
Hence parts (1) and (2) are true.

To show part (3),
suppose, for the subgraph induced by the chords of $C$ in $P(D)$, $T$ is its component containing $uv$.
Then take a vertex $v_1$ distinct from $u$ in $T$.
Then there exists a path $P=v_1 \cdots v_{t} u$ in $T$ and each edge in $P$ is a chord in $C$.
Therefore each edge in $P$ is a cared edge by Lemma~\ref{lem:char-cycle}(4).
Let $y$ be a vertex taking care of $uv_t$.
Then $y \in V(D) - V(H)$ by
Lemma~\ref{lem:char-cycle}(4).
Since
$w$ is the only out-neighbor in $V(D)- V(H)$ of $u$,
$w=y$ and so $w$
is an out-neighbor of $v_t$.
If $t\geq 2$, then, by applying a similar argument for the chord $v_tv_{t-1}$,
$w$ takes care of $v_tv_{t-1}$ and
 $w$ is the only out-neighbor in $V(D)- V(H)$ of $v_t$ by parts (1) and (2). 
 Therefore $w$ is an out-neighbor of $v_{t-1}$ if $t \geq 2$.
 We repeat this process until we conclude that $w$ is an out-neighbor of $v_1$.
 Therefore part (3) is true.
\end{proof}

\begin{Cor} \label{cor:non-chord-2-out}
Let $H$ be a hole with length $l \geq 5$ in the underlying graph of an $(i,2)$ digraph $D$ and $C$ be the cycle obtained from $H$. Then each vertex on $C$ having two out-neighbors in $V(H)$ is not incident to any chord of $C$.
\end{Cor}
\begin{proof}
Let $u$ be a vertex on $C$ having two out-neighbors in $V(H)$.
Suppose, to the contrary, that $u$ is incident a chord $uv$ of $C$.
Then there exists a vertex $w$ in $V(D)-V(H)$ 
such that $w$ is a common out-neighbor of $u$ and $v$ by Proposition~\ref{prop:chord_property}(2).
Thus $u$ has at least three out-neighbors, which contradicts the fact that $D$ is an $(i,2)$ digraph.
\end{proof}

To prove one of our main theorems, we need one more result.
\begin{Lem} [\cite{eoh2021chordal}]\label{lem:induced-path-hole}
Given a graph, $G$, and a cycle, $C$, of $G$ with length at least four, suppose
that a section, $Q$, of $C$ forms an induced path of $G$ and contains a path, $P$, with length at
least two, none of whose internal vertices is incident to a chord of $C$ in $G$. Then $P$ can
be extended to a hole $H$ in $G$ so that $V(P) \subsetneq V(H)  \subseteq V(C)$ and $H$ contains a vertex on $C$ that is not on $Q$.
\end{Lem}

Given a vertex subset $X$ of a graph $G$,
a {\it maximum clique in $X$} means a clique in $G[X]$ whose size is the maximum among the cliques in $G[X]$.

%acyclic이고 outdegree가 2 제한만 있으면 성립함.

\begin{Thm}\label{thm:maximumclique}
Let $D$ be an $(i,2)$ digraph and $C$ be the cycle of length $l\geq 4$ in $P(D)$ obtained from a hole in $U(D)$.
If a maximum clique in $V(C)$ has size at most 
$\lfloor\frac{ l -1 }{2}\rfloor$,
 then the subgraph of $P(D)$ induced by $V(C)$ has a hole.
\end{Thm}

\begin{proof}
Suppose that a maximum clique $K$ in $V(C)$ has size at most $\lfloor\frac{ l -1 }{2}\rfloor$.
If $C$ has no chord, then $C$ is a hole and so we are done.
Suppose that $C$ has a chord.
If a maximum clique has size at least three in $V(C)$, then the clique must contain a chord of $C$.
Otherwise, each chord is a maximum clique itself.
Therefore
we may assume that
$\{u,v\} \subseteq V(K)$ for a chord $uv$ of $C$.
%where->for로 변경
We note that
\[ |V(C)| - |V(K)| \geq  l - \left\lfloor\frac{ l -1 }{2}\right \rfloor = \left\lceil\frac{ l+1 }{2}\right \rceil > \frac{l}{2}.
\]
%Since $|V(K)| \leq \lfloor\frac{ l -1 }{2}\rfloor$, 
%$l- \lfloor\frac{ l -1 }{2}\rfloor \geq \frac{ l  }{2}$ and 
Therefore 
there exist two consecutive vertices $x_1$ and $x_2$ on $C$ each of which does not belong to $V(K)$.
Starting from $x_1$ (resp.\ $x_2$), we traverse the $(x_1,x_2)$-section (resp.\ the $(x_2,x_1)$-section) of $C$ that is not the edge $x_1x_2$ until we first meet a vertex $y$ (resp.\ $z$) belonging to $V(K)$.
Then the $(y,x_1)$-section obtained in this way, the edge $x_1x_2$, and the $(x_2,z)$-section obtained in this way form the $(y,z)$-section $Q$ of $C$ such that $y$ and $z$ are the only vertices belonging to $V(K)$. 
Since $x_1$ and $x_2$ are contained in $Q$,
$Q$ has length at least $3$.
Let $Q=v_0v_1v_2\cdots v_t$ where $v_0=y$ and $v_t=z$ for an integer $t\geq 3$.
Then $v_i \notin V(K)$ for each $1\leq i\leq t-1$.
If $v_0=v_t$, then $V(Q)=V(C)$ and so $V(K) \cap V(C)=\{v_0\}$, which contradicts that the existence of the chord $uv$.
Therefore \[v_0 \neq v_t.\]
Since $\{v_0,v_t\} \subseteq V(K)$, $v_0v_t$ is an edge in $P(D)$.
If $v_0v_t$ is not a chord of $C$, then $C= v_0v_1\cdots v_tv_0$ and so, by the choice of $Q$, $V(K)$ does not contain any chord, which is a contradiction.
Therefore $v_0v_t$ is a chord of $C$.

Let $T$ be the component containing $v_0v_t$ in the induced subgraph by the chords of $C$.
We note that
\begin{enumerate}
\item[($\star$)]
 any vertex in $T$ cannot be joined to a vertex on $C- T$ by a chord of $C$.
\end{enumerate}
Let $C_1$ be the cycle obtained from adding
$v_0v_t$ to $Q$. 
Suppose $V(T)=\{v_0,v_t\}$.
%edge 하나라서 concatenating대신에 adding으로 표현
Then $C_1$ has length at least four and, by ($\star$), $P_1:=v_1v_0v_t$ is an induced path.
Since $V(T)=\{v_0,v_t\}$, $v_0$ is not
incident to any chord of $C$ except $v_0v_t$ and so $v_0$ is not incident to any chord of $C_1$.
Now we suppose
\[V(T) \neq \{v_0,v_t\}.\]
Then $|V(T)| \geq 3$.
Furthermore, since $l \geq 4$, 
the hole in $U(D)$ containing the vertices on $C$ has length at least $5$.
Therefore $V(T)$ forms a clique in $P(D)$ by Proposition~\ref{prop:chord_property}(3) and so
$|V(K)|\geq 3$ by the maximality.
By the choice of $Q$,
$K$ contains a vertex on the $(v_0,v_t)$-section, say $L$, of $C$ other than $Q$. %distinct from => other than
Since $v_0v_t$ is a chord of $C$, $L$ has length at least two.

{\it Case 1}. $L$ has length $2$.
Let $w$ be the internal vertex on $L$.
Then $V(K)=\{v_0,w,v_t\}$ and so, by the maximality of $K$, $|V(T)|=3$.
Therefore $T=\{v_0,v_j,v_t\}$ for some $j\in \{1,\ldots,t-1\}$.
If $l\leq 6$,
$\lfloor\frac{ l -1 }{2}\rfloor < 3$, which contradicts the fact that $K$ has size $3$.
Therefore $l \geq 7$ and so \[ t \geq 5.\]
If $j=1$ or $2$, then, by ($\star$), 
$P_2:=v_tv_jv_{j+1}$ is an induced path and $v_j$ is not incident to any chord of $C_2$
where $C_2$ is the cycle
of length $t-j+1$ obtained from adding $v_jv_t$ to the $(v_j,v_t)$-section of $Q$.
If $j\geq 3$, then, by ($\star$), $P_3:=v_1v_0v_j$ is an induced path and $v_0$ is not incident to any chord of $C_3$ where $C_3$ is the cycle of length $j+1$ obtained from adding $v_0v_j$ to the $(v_0,v_j)$-section of $Q$.
We note that $t-j+1\geq 4$ if $j=1$ or $2$ and $j+1\geq 4$ for $j\geq 3$.
Therefore each of $C_2$ and $C_3$ has length at least $4$.

{\it Case 2}.  $L$ has length at least $3$.
Then, for each vertex $x$ on $K$, $v_0x$ or $v_tx$ is a chord of $C$.
Therefore $V(K) \subseteq V(T)$ and so, by the maximality, $V(K)=V(T)$.
Thus, by ($\star$),
$P_1=v_1v_0v_t$ is an induced path and $v_0$ is not incident to any chord of the cycle $C_1$.

For each $1\leq i \leq 3$, by applying Lemma~\ref{lem:induced-path-hole} to $P_i$ and $C_i$,
we may conclude that $P_i$ can be extended to a hole in $P(D)$ whose vertices are on $C_i$.
Since the cycles $C_1,C_2$, and $C_3$ are contained in $V(C)$, the subgraph of $P(D)$ induced by $V(C)$ has a hole and so the statement is true.
\end{proof}

\begin{Lem} \label{lem:chord-connectivity}
Let $G$ be a graph and $C$ be a cycle of $G$.
Suppose that there exists a maximum clique $K$ of size at least four in $V(C)$.
If $C$ has length at least five,
then, for each pair of vertices in $K$,
there is a path between them consisting of only chords of $C$.
\end{Lem}
\begin{proof}
Suppose that $C$ has length at least five.
Let $K^*$ be the graph obtained from $K$ by deleting edges of $C$ in $K$.
Since $|V(K^*)|=|V(K)|$,
there are at least four vertices in $V(K^*)$ and we take four vertices $v_1,v_2,v_3,v_4$ in $V(K^*)$. 
Let $T$ be the subgraph induced by $\{v_1,v_2,v_3,v_4\}$ in $K^*$.
Since $C$ has length at least five, at most three edges of $C$ were deleted from the clique on $\{v_1,v_2,v_3,v_4\}$ to obtain $T$.
Even if three edges were deleted, $T$ is isomorphic to a path of length $3$.
 Thus $T$ is connected.
 Since $v_1,v_2,v_3$, and $v_4$ were arbitrarily chosen from $V(K^*)$, we conclude that $K^*$ is connected and so the statement is true.
\end{proof}

\begin{Lem} \label{lem:maximum-size}
Let $H$ be a hole with length $l \geq 5$ in the underlying graph of an $(i,2)$ digraph $D$ and $C$ be the cycle in $P(D)$ obtained from $H$.
If $i \geq 3$,
then a maximum clique in $V(C)$ has size at most $i$ in $P(D)$.
\end{Lem}
\begin{proof}
To the contrary, we suppose that $i \geq 3$ and there exists a clique $K$ in $V(C)$ of size at least $i+1$.
Then $|V(K)| \geq 4$.
We first assume that $C$ has length at least five.
Then, by Lemma~\ref{lem:chord-connectivity},
for each pair of vertices in $K$, there is a path between them consisting of only chords of $C$.
Therefore, by Proposition~\ref{prop:chord_property}(3),
there exists a common out-neighbor $w$ of the vertices in $K$.
Thus $w$ has indegree at least $i+1$, which contradicts the fact that $D$ is an $(i,2)$ digraph.
Hence $C$ has length at most $4$.
Then, since $|V(K)| \geq 4$, $C$ has length $4$ and so $|V(C)|=|V(K)|=4$.
Thus each vertex on $C$ is incident to a chord.
However, since $|\Gamma_H|\geq 1$, 
there is a vertex on $C$ having two out-neighbors in $V(H)$ by Lemma~\ref{lem:char-cycle}(3). Thus the vertex on $C$ is not incident to any chord by Corollary~\ref{cor:non-chord-2-out} and so we reach a contradiction.
Hence the statement is true.
\end{proof}

Given two sets $A$ and $B$, we write $A \sqcup B$ to mean $A\cup B$ where $A\cap B=\emptyset$.
\begin{Thm} \label{thm:determinant-hole}
Let $H$ be a hole with length $l \geq 5$ in the underlying graph of an $(i,2)$ digraph $D$ with $i \geq 3$ and $C$ be the cycle in $P(D)$ of length at least four obtained from $H$ by $\Gamma_H$.
If
\[
l-|\Gamma_H| \geq 2i+1 \quad \text{or} \quad  l < 3 |\Gamma_H|, \]
then the subgraph of the phylogeny graph of $D$ induced by $V(C)$ has a hole.
\end{Thm}

\begin{proof}
By the definition of $C$, $|V(C)| = l-|\Gamma_H|$.
Since $i \geq 3 $ by the assumption, each maximum clique in $V(C)$ has size at most $i$ by Lemma~\ref{lem:maximum-size}.

We first suppose $l-|\Gamma_H| \geq 2i+1$.
Then $\frac{|V(C)|-1}{2} \geq i$.
Since $|V(C)| \leq l-1$ by Lemma~\ref{lem:char-cycle}(1),
\[\frac{|V(C)|-1}{2}  \leq \frac{l-2}{2} \leq \left\lfloor \frac{l-1}{2} \right\rfloor.
\]
Therefore
\[\left\lfloor \frac{l-1}{2} \right\rfloor \geq i \]
and so, by Theorem~\ref{thm:maximumclique},  the statement is true.

Now we suppose
\begin{equation} \label{eq:determinant-1}
l < 3 |\Gamma_H|.
\end{equation}
Let \[A=\{v \in V(C) \mid |N^+(v) \cap V(H) | =2 \} \] and\[ B=\{v \in V(C) \mid  |N^+(v) \cap V(H) | =1\}.\]
Take a vertex $v$ on $C$.
Then, by Lemma~\ref{lem:char-cycle}(2),
$N^+(v) \cap V(H) \neq \emptyset$.
Since $D$ is an $(i,2)$ digraph, $|N^+(v) \cap V(H) |\leq 2 $ and so $v \in A \cup B$.
Therefore $V(C) = A\sqcup B$.
Then, since $|A|=|\Gamma_H|$ by Lemma~\ref{lem:char-cycle}(3),
$|B|=|V(C)|-|A|=l-2|\Gamma_H|$.
Thus $|B| < |A|$ by~\eqref{eq:determinant-1}.
Hence
there exist two consecutive vertices $u$ and $v$ on $C$
that belong to $A$.
We take the section $Q:=uvw$ of $C$.
Since $\{u, v\} \subseteq A$, neither $u$ nor $v$ is incident to any chord of $C$ by Corollary~\ref{cor:non-chord-2-out}.
Therefore $Q$ is an induced path of length two.
Thus $Q$ can be extended to a hole in $P(D)$ whose vertices are on $C$ by Lemma~\ref{lem:induced-path-hole}.
Thus the statement is true.
\end{proof}

In preparation for the proof of Theorem~\ref{thm:main}, we need the notion of perfect elimination orderings.
 A {\it perfect elimination ordering} in a graph is an ordering of the vertices of the graph such that, for each vertex $v$, $v$ and the neighbors of $v$ that come after $v$ in the order form a clique.
  It is well-known that a graph is chordal if and only if it has a perfect elimination ordering.
  A {\it simplicial vertex} is one whose neighbors form a clique.
  
  From now on, we write $u \to v$ (resp.\ $u \not \to v$) to represent ``$(u,v)$ is (resp.\ is not) an arc of a digraph".
\begin{proof}[Proof of Theorem~\ref{thm:main}]
Suppose $l \geq 3i+1$.
If $i=2$, then the statement is true by Theorem~\ref{thm:(2,2)-hole}.
Now we assume $i\geq 3$.
Let $C$ be the cycle obtained from $H$ by $\Gamma_H$.
Since $i \geq 3$,  $\frac{l}{2}\geq \frac{3i+1}{2} \geq 5$.
Then, since $|V(C)|\geq  l-\lfloor \frac{l}{2} \rfloor \geq \frac{l}{2} $ by Lemma~\ref{lem:char-cycle}(1), $C$ has length at least five.
To the contrary, suppose the subgraph of $P(D)$ induced by $V(H)$ does not have a hole. Then by Theorem~\ref{thm:determinant-hole}, \[
l-|\Gamma_H| \leq 2i \quad \text{and} \quad l \geq 3 |\Gamma_H|. \]
Then $3|\Gamma_H|-|\Gamma_H|\leq l-|\Gamma_H|\leq 2i $ and so we obtain
$|\Gamma_H| \leq i$.
Since $l \geq 3i+1$, $l- |\Gamma_H| \geq 2i+1$, a contradiction.
Therefore the subgraph of $P(D)$ induced by $V(H)$ has a hole and so the statement is true.

To show the ``further" part,
we consider an $(i,2)$ digraph $D$ with the vertex set \[V(D)=\{u,v_{0,1},v_{0,2},v_{0,3},v_{1,1},v_{1,2},v_{1,3},\ldots,
 v_{i-1,1},v_{i-1,2},v_{i-1,3}\}\]
 and the arc set
 \[
 A(D)=\{(v_{j,1},v_{j,2}),(v_{j,2},v_{j,3}),(v_{j,1},v_{j-1,3}),(v_{j,2},u)\mid 0\leq j \leq i-1 \}\]
 (each subscript of the vertices in $D$ is reduced to modulo $i$ and see the digraph given in Figure~\ref{fg:digraph_remark} for an illustration).
 Then $V(D) - \{u\}$ forms a hole $H$ of length $3i$ in $U(D)$.
 Since $v_{j-1,2}\to v_{j-1,3}$ and $v_{j,1} \to v_{j-1,3}$ for each $0\leq j \leq i-1$, $C=v_{0,1}v_{0,2}v_{1,1}v_{1,2}\cdots v_{i-1,1}v_{i-1,2}v_{0,1}$ is the cycle
obtained from $H$ in $P(D)$.
Since $u$ is a common out-neighbor of any pair in $K:=\{v_{j,2}\mid 0\leq j \leq i-1\}$, $K$ forms a clique in $P(D)$.
We can check that
for each $0 \leq j \leq i-1$, in $P(D)$,
 \begin{align*}
&N(u)=K,\quad N(v_{j,1})=\{v_{j-1,3},v_{j-1,2},v_{j,2}\},
\\&N(v_{j,2})=\{u,v_{j,1},v_{j,3},v_{j+1,1}\}\sqcup K,\quad \text{and} \quad N(v_{j,3})=\{v_{j,2},v_{j+1,1}\}.
 \end{align*}
 We note that, for each $0\leq j \leq i-1$, $v_{j,3}$ is a simplicial vertex in $P(D)-u$ and $v_{j,1}$ is a  simplicial vertex in $P(D) - \{u,v_{0,3},v_{1,3},\ldots,v_{i-1,3}\}$.
Therefore \[u,v_{0,3},v_{1,3},\ldots,v_{i-1,3},v_{0,1},v_{1,1},\ldots,v_{i-1,1},v_{0,2},v_{1,2},\ldots,v_{i-1,2}\]
is a perfect elimination ordering and so $P(D)$ is chordal.
Then, since $|V(H)|=3i$,
we conclude that the desired bound $3i+1$ is achieved by $D$.
\end{proof}
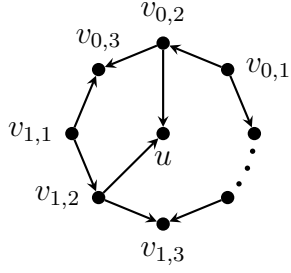
\begin{figure}
\begin{center}
  \begin{tikzpicture}[auto,thick,scale=0.8]
    \tikzstyle{player}=[minimum size=5pt,inner sep=0pt,outer sep=0pt,fill,color=black, circle]
    \tikzstyle{player2}=[minimum size=2pt,inner sep=0pt,outer sep=0pt,fill,color=black, circle]
    \tikzstyle{source}=[minimum size=5pt,inner sep=0pt,outer sep=0pt,ball color=black, circle]
    \tikzstyle{arc}=[minimum size=5pt,inner sep=1pt,outer sep=1pt, font=\footnotesize]
    \path (45:1.5cm)   node [player]  (v_{0,1}) [label=right:$v_{0,1}$] {};
    \path (90:1.5cm)     node [player]  (v_{0,2})[label=above:$v_{0,2}$] {};
    \path (135:1.5cm)   node [player]  (v_{0,3})[label=above:$v_{0,3}$]  {};
    \path (180:1.5cm)   node [player]  (v_{1,1})[label=left:$v_{1,1}$ ]  {};
    \path (225:1.5cm)   node [player]  (v_{1,2})[label=left:$v_{1,2}$ ]  {};
    \path (270:1.5cm)   node [player]  (v_{1,3})[label=below:$v_{1,3}$ ]  {};
    \path (0:0cm)   node  [player]   (u)[label=below:$u$ ]  {};
    \path (315:1.5cm)   node [player]  (dot1)[label=right:$$] {};
    \path (326.25:1.5cm)   node [player2]  (dot1_1)[label=right:$$] {};
    \path (337.5:1.5cm)   node [player2]  (dot2)[label=right:$$] {};
    \path (348.75:1.5cm)   node [player2]  (dot2_1)[label=right:$$] {};
    \path (360:1.5cm)   node [player]  (dot3)[label=right:$$] {};

  \draw[black,thick,-stealth] (v_{0,1}) to (v_{0,2});
   \draw[black,thick,-stealth] (v_{0,2}) to (v_{0,3});
   \draw[black,thick,-stealth] (v_{0,2}) to (u);
   \draw[black,thick,-stealth] (v_{1,1}) to (v_{0,3});
    \draw[black,thick,-stealth] (v_{1,1}) to (v_{1,2});
    \draw[black,thick,-stealth] (v_{1,2}) to (v_{1,3});
    \draw[black,thick,-stealth] (v_{1,2}) to (u);
    \draw[black,thick,-stealth] (v_{0,1}) to (dot3);
    \draw[black,thick,-stealth] (dot1) to (v_{1,3});
    \end{tikzpicture}
    \end{center}
\caption{The digraph $D$ in the proof for the ``further" part of Theorem~\ref{thm:main}}
\label{fg:digraph_remark}
 \end{figure}
\section{Forbidden subgraphs for phylogeny graphs of degree bounded digraphs}
\begin{Prop} \label{prop:bounded-clique}
Let $D$ be an $(i,j)$ digraph and $N$ be a subset of the neighbors of some vertex in $P(D)$.
If any $k$ vertices in $N$ do not form a clique in $P(D)$ for some positive $k$, then $|N|\leq (k-1)(j+1)$.

\end{Prop}

\begin{proof}
Let $u$ be a vertex such that $N$ is a set of its neighbors in $P(D)$.
Suppose that any $k$ vertices in $N$ do not form a clique in $P(D)$ for some positive integer $k$.

By the definition of phylogeny graph,
\begin{equation}\label{eq:prop:bounded-clique1}
 N(u) = \left(\bigcup_{v \in N^+(u)} N^-(v) -\{u\}\right) \cup N^+(u) \cup N^-(u).
\end{equation}
Take a vertex $v$ in $D$.
Then $N^-(v)\cup\{v\}$ forms a clique in $P(D)$ and so $(N^-(v)\cup\{v\})\cap N$ is an empty set or forms a clique in $P(D)$.
Thus, by our assumption,
\begin{equation}\label{eq:prop:bounded-clique2}
|N^-(v) \cap N| \leq |(N^-(v)\cup\{v\}) \cap N| \leq k-1.\end{equation}
Furthermore, if $v\in N$, then
$|N^-(v) \cap N| < |(N^-(v)\cup\{v\}) \cap N|$ and so
\begin{equation}\label{eq:prop:bounded-clique3}
|N^-(v) \cap N| \leq k-2.\end{equation}
We note that $N(u) \cap N = N$ and $u \notin N$. 
Then, by \eqref{eq:prop:bounded-clique1}, \eqref{eq:prop:bounded-clique2}, and \eqref{eq:prop:bounded-clique3},
\begin{align*}
|N| &\leq
\sum_{{v \in N^+(u)\cap N}} |N^-(v)\cap N|
+\left(
\sum_{v \in N^+(u)- N} |N^-(v)\cap N|\right)
+
|N^+(u)\cap N| + |N^-(u) \cap N|
\\
& \leq (k-2) \cdot |N^+(u)\cap N|  + (k-1) \cdot |N^+(u)- N|
+ |N^+(u)\cap N| + (k-1) \\
& =
(k-1)(|N^+(u)\cap N| + |N^+(u)- N|+1)= (k-1) (|N^+(u)| +1)\leq (k-1)(j+1).\qedhere 
\end{align*}
\end{proof}

We say that a graph $G$ is {\it $(i,j)$ phylogeny-realizable through an $(i,j)$ digraph} if it is an induced subgraph of a phylogeny graph of an $(i,j)$ digraph. When no confusion is likely to arise, we simply say $G$ is $(i,j)$ realizable.

Let $G$ be a graph.
For vertices $x$ and $y$ in $G$,
the length of a shortest path from $x$ to $y$ is called the {\it distance} between $x$ and $y$ and denoted by $d_G(x,y)$.
A {\it center} is a vertex $u$ such that $\max\{ \text{$d_G(u, v)$} : v \in  V(G)\}$ is as small as possible.

\begin{Prop} \label{prop:forbidden_star}
If an $(i,j)$ phylogeny graph contains an induced subgraph $H$ isomorphic to $K_{1,l}$ for some positive integer $l$,
then $l\leq j+1$.
 Furthermore, $H\cong K_{1,j+1}$ is $(i,j)$ realizable. 
\end{Prop}

\begin{proof}
We suppose that
an $(i,j)$ phylogeny graph $P(D)$ contains 
an induced subgraph $H$ isomorphic to $K_{1,j+2}$.
Let $u$ be the center of $H$.
Then $V(H)- \{u\}$ is a subset of $N(u)$ such that any two vertices in $V(H)- \{u\}$ do not form a clique in $P(D)$.
Therefore $|V(H)- \{u\}| \leq j+1$
by Proposition~\ref{prop:bounded-clique}, which is a contradiction.

To show that $H\cong K_{1,j+1}$ is $(i,j)$ realizable,
let $D$ be a digraph with
the vertex set
\[V(D) =\{u,v,w_1,\ldots,w_{j}\} \]
and the arc set
\[A(D)=\{ (u,v)\} \cup \{(v,w_k) \mid 1\leq k \leq j\}.\]
Then we can check that $D$ is a $(1,j)$ digraph and $P(D)$ is isomorphic to $K_{1,j+1}$ with the center $v$ and so the statement is true.\end{proof}

\begin{Lem}\label{lem:K_1,j+1}
If an $(i,j)$ phylogeny graph contains 
an induced subgraph $H$ isomorphic to $K_{1,j+1}$
with the center $v$, 
then $|N^-(v) \cap V(H)|=1$.
\end{Lem}
\begin{proof}
Suppose that, for an $(i,j)$ digraph $D$, $P(D)$ contains an induced subgraph $H$ with the center $v$ isomorphic to $K_{1,j+1}$.
Since $H$ is triangle-free, $|N^-(v) \cap V(H)| \le 1$.
To the contrary, suppose that $|N^-(v) \cap V(H)| =0$.
Then, if all the edges in $H$ are cared edges,
$v$ has at least $j+1$ out-neighbors, which contradicts the fact that $D$ is an $(i,j)$ digraph.
Therefore $H$ has at least one edge in $U(D)$.
Let $|N^+(v)\cap V(H)|=k$ for a positive integer $k$.
Then
\begin{equation}\label{eq:lem:K_1,j+1_1}
|N^+(v)-V(H)|\leq j-k	
\end{equation}
 since $v$ has at most $j$ out-neighbors.
Moreover, there are $j+1-k$ cared edges incident to $v$.
Let $uv$ be a cared edge.
Then, by the definition of cared edges,
$u$ and $v$ have a common out-neighbor $w$.
 Since $H$ is triangle-free, 
 $w$ belongs to $N^+(v)- V(H)$
 and $N^-(w)\cap V(H)=\{u,v\}$.
 Since the cared edge $uv$ was arbitrarily chosen, we may conclude that the number of cared edges, which equals $j+1-k$, is less than or equal to $|N^+(v)-V(H)|$, which contradicts \eqref{eq:lem:K_1,j+1_1}.
\end{proof}

\begin{Prop}\label{prop:forbidden_bipartite}
If an $(i,j)$ phylogeny graph contains an induced subgraph $H$ isomorphic to $K_{m,n}$ for some positive integers $m$ and $n$,
then $m \leq j+1$ and $n\leq j+1$ where the equalities cannot hold simultaneously. Furthermore, $H\cong K_{j+1,j}$ is $(i,j)$ realizable unless $i=1$.
\end{Prop}

\begin{proof}
We suppose that there exists an $(i,j)$ phylogeny graph $P(D)$ containing an induced subgraph $H$ isomorphic to $K_{m,n}$ for some positive integers $m$ and $n$.
Take a vertex $v$ in the partite set of $H$ with size $m$.
Let $X$ be the other partite set of $H$.
Then $(\{v\},X)$ is the bipartition of a subgraph isomorphic to $K_{1,n}$.
Thus $n \leq j+1$ by Proposition~\ref{prop:forbidden_star}.
By symmetry, we conclude $m \leq j+1$.

To show that either $m<j+1$ or $n < j+1$ by contradiction, suppose $m=n=j+1$.
Let $v$ be a vertex in $H$ as before.
Then there exists a subgraph $H_v$ in $H$ isomorphic to $K_{1,j+1}$ such that $v$ is the center of $H_v$.
Therefore, by Lemma~\ref{lem:K_1,j+1}, $v$ has one in-neighbor in the subdigraph induced by $V(H_v)$.
Then, since $H_v$ is a subgraph of $H$, $v$ has one in-neighbor in $V(H)$.
Since $v$ was arbitrarily chosen from $H$,
each vertex in the subdigraph induced by $V(H)$ has one in-neighbor in $V(H)$.
Take a vertex $v_1$ in $V(H)$.
Then there exists an in-neighbor $v_2$ in $V(H)$.
We may repeat this process until we obtain a directed cycle, which contradicts the fact that $D$ is acyclic.

Now we show that $H\cong K_{j+1,j}$ is $(i,j)$ realizable.
We construct an $(i,j)$ digraph whose phylogeny graph contains an induced subgraph isomorphic to $K_{j+1,j}$ for an integer $i\geq 2$.
Let $D$ be a digraph with the vertex set
\[
V(D)=\{u_1,u_2\ldots,u_{j+1}, v_1,v_2,\ldots,v_{j}\} \cup \{w_{l,m} \mid 1\leq l,m \leq j\},\]
and the arc set
\begin{align*}
A(D)&=\{ (u_l,v_l)\mid 1\leq l \leq j\} \cup \{(v_l,w_{l,m}) \mid 1\leq l,m\leq j \} \\
&\cup  \{ (u_{l},w_{m,l} ) \mid 1\leq l,m \leq j, \ l\neq m\} \cup
 \{ (u_{j+1},w_{l,l}) \mid 1\leq l \leq j\}
\end{align*}
 (see the $(2,2)$ digraph whose phylogeny graph having an induced subgraph isomorphic to $K_{3,2}$ with the bipartition $(\{u_1,u_2,u_3\},\{v_1,v_2\})$ given in Figure \ref{fig:$(i,2)$-digraph-bipartite} for an illustration).
 Then we can check $D$ is an $(i,j)$ digraph and 
\[( \{u_1,u_2\ldots,u_{j+1}\}, \{ v_1,v_2,\ldots,v_{j}\} )\]
is the bipartition of a subgraph isomorphic to $K_{j+1,j}$ in $P(D)$.
\end{proof}
\begin{figure}
\begin{center}
\begin{tikzpicture}[auto,thick, scale=1]
    \tikzstyle{player}=[minimum size=5pt,inner sep=0pt,outer sep=0pt,color=black,fill,draw,circle]
    \tikzstyle{source}=[minimum size=5pt,inner sep=0pt,outer sep=0pt,ball color=black, circle]
    \tikzstyle{arc}=[minimum size=5pt,inner sep=1pt,outer sep=1pt, font=\footnotesize]
   \path (0,3.5)   node [player] [label=above:$w_{1,1}$] (w11) {};
    \path (1,3.5)   node [player] [label=above:$w_{1,2}$] (w12) {};
     \path (0,0.5)   node [player] [label=below:$w_{2,1}$] (w21) {};
    \path (1,0.5)   node [player] [label=below:$w_{2,2}$] (w22) {};
      \path (-1,2.7)   node [player] [label=left:$u_1$] (u1) {};
\path (-1,2)   node [player] [label=left:$u_2$] (u2) {};
\path (-1,1.3)   node [player] [label=left:$u_3$] (u3) {};
     \path (2,2.3)   node [player] [label=right:$v_1$] (v1) {};
\path (2,1.4)   node [player] [label=right:$v_2$] (v2) {};
\draw[black,thick,-stealth] (u1) - + (v1);
\draw[black,thick,-stealth] (u2) - + (v2);

\draw[black,thick,-stealth] (v1) - + (w11);
\draw[black,thick,-stealth] (v1) - + (w12);
\draw[black,thick,-stealth] (v2) to (w21);
\draw[black,thick,-stealth] (v2) to  (w22);

\draw[black,thick,-stealth] (u1)  to (w21);
\draw[black,thick,-stealth] (u2) - + (w12);
\draw[black,thick,-stealth] (u3) - + (w11);
\draw[black,thick,-stealth] (u3) - + (w22);
    \end{tikzpicture}
    \end{center}
\caption{A $(2,2)$ digraph whose phylogeny graph contains $K_{3,2}$ as an induced subgraph}
\label{fig:$(i,2)$-digraph-bipartite}
\end{figure}
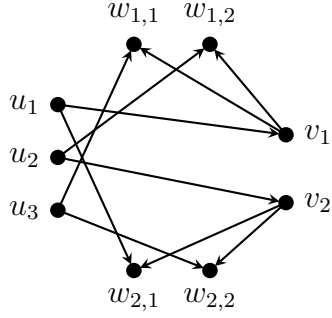

For $i=1$ or $j=1$,
an $(i,j)$ phylogeny graph is completely characterized by the following theorems.

\begin{Thm} [\cite{eoh2021chordal}] \label{thm:eoh-char-(1,j)}
 For a positive integer $j$, a graph is a $(1,j)$ phylogeny graph if and only if it is a forest with the maximum degree at most $j+1$.
\end{Thm}
Given a graph $G$, we denote the size of a maximum clique in $G$ by $\omega(G)$. 
The {\it clique graph} of a graph $G$, denoted by $K(G)$, is a simple graph such that (i) every vertex of $K(G)$ represents a maximal clique of $G$ and (ii) two vertices of $K(G)$ are adjacent when they share at least one vertex in common in G.

\begin{Thm} [\cite{eoh2021chordal}] \label{thm:eoh-char-(i,1)}
 For a positive integer $i$, a graph is an $(i,1)$ phylogeny graph if and only if it is a diamond-free chordal graph with $\omega(G) \leq i+1$ and its clique graph is a forest.
\end{Thm}

\begin{figure}
\begin{center}
\begin{tikzpicture}[auto,thick, scale=1.2]
    \tikzstyle{player}=[minimum size=5pt,inner sep=0pt,outer sep=0pt,color=black,fill,draw,circle]
    \tikzstyle{source}=[minimum size=5pt,inner sep=0pt,outer sep=0pt,ball color=black, circle]
    \tikzstyle{arc}=[minimum size=5pt,inner sep=1pt,outer sep=1pt, font=\footnotesize]
   \path (0:0cm)   node [player] [label=right:$u$] (u) {};
    \path (225:1cm)   node [player] [label=left:$w_4$] (w4) {};
    \path (180:1cm)   node [player]
    [label=left:$w_3$]
     (w3) {};
    \path (135:1cm)    node [player]
     [label=left:$w_2$]
     (w2) {};
    \path (90:1cm)   node [player]
    [label=above:$w_1$]
       (w1) {};
    \path (45:1cm)   node [player]
     [label=right:$v_4$]
       (v4) {};
    \path (0:1cm)   node [player]
     [label=right:$v_3$]
      (v3) {};
      \path (-45:1cm)   node [player]
     [label=below:$v_2$]
      (v2) {};
       \path (270:1cm)   node [player]
     [label=below:$v_1$]
      (v1) {};
    \draw (270:2cm) node (name) {$D_1$};

\draw[black,thick,-stealth] (u) - + (v4);
\draw[black,thick,-stealth] (u) - + (w2);
\draw[black,thick,-stealth] (u) - + (w4);

\draw[black,thick,-stealth] (v1) - + (u);
\draw[black,thick,-stealth] (v2) - + (u);
\draw[black,thick,-stealth] (v2) - + (v3);
\draw[black,thick,-stealth] (v3) - + (v4);
\draw[black,thick,-stealth] (v4) - + (w1);
\draw[black,thick,-stealth] (w1) - + (w2);
\draw[black,thick,-stealth] (w2) - + (w3);
\draw[black,thick,-stealth] (w3) - + (w4);
    \end{tikzpicture}
\hspace{3em}
\begin{tikzpicture}[auto,thick, scale=1.2]
    \tikzstyle{player}=[minimum size=5pt,inner sep=0pt,outer sep=0pt,color=black,fill,draw,circle]
    \tikzstyle{source}=[minimum size=5pt,inner sep=0pt,outer sep=0pt,ball color=black, circle]
    \tikzstyle{arc}=[minimum size=5pt,inner sep=1pt,outer sep=1pt, font=\footnotesize]
   \path (0:0cm)   node
   [player]
[label={[xshift=0.38cm, yshift=-0.18cm]:$u$}]
 (u) {};
    \path (225:1cm)   node [player] [label=left:$w_4$] (w4) {};
    \path (180:1cm)   node [player]
    [label=left:$w_3$]
     (w3) {};
    \path (135:1cm)    node [player]
     [label=left:$w_2$]
     (w2) {};
    \path (90:1cm)   node [player]
    [label=above:$w_1$]
       (w1) {};
    \path (45:1cm)   node [player]
     [label=right:$v_4$]
       (v4) {};
    \path (0:1cm)   node [player]
     [label=right:$v_3$]
      (v3) {};
      \path (-45:1cm)   node [player]
     [label=below:$v_2$]
      (v2) {};
       \path (270:1cm)   node [player]
     [label=below:$v_1$]
      (v1) {};
    \draw (270:2cm) node (name) {$P(D_1)$};

\draw (u) -- (v1);
\draw (u) -- (v2);
\draw (u) -- (v3);
\draw (u) -- (v4);
\draw (u) -- (w1);
\draw (u) -- (w2);
\draw (u) -- (w3);
\draw (u) -- (w4);

\draw (v1) -- (v2);
\draw (v2) -- (v3);
\draw (v3) -- (v4);
\draw (v4) -- (w1);
\draw (w1) -- (w2);
\draw (w2) -- (w3);
\draw (w3) -- (w4);
    \end{tikzpicture}

%2nd picture
\begin{tikzpicture}[auto,thick, scale=1.2]
    \tikzstyle{player}=[minimum size=5pt,inner sep=0pt,outer sep=0pt,color=black,fill,draw,circle]
    \tikzstyle{source}=[minimum size=5pt,inner sep=0pt,outer sep=0pt,ball color=black, circle]
    \tikzstyle{arc}=[minimum size=5pt,inner sep=1pt,outer sep=1pt, font=\footnotesize]
   % \draw (22.5:0.35cm) node (name) {$u$};
   \path (0:0cm)   node [player] [label=right:$u$] (u) {};
    \path (225:1cm)   node [player] [label=left:$w_4$] (w4) {};
    \path (247.5:1.5cm)   node [player] [label=below:$x$] (x) {};

    \path (180:1cm)   node [player]
    [label=left:$w_3$]
     (w3) {};
    \path (135:1cm)    node [player]
     [label=left:$w_2$]
     (w2) {};
    \path (90:1cm)   node [player]
    [label=above:$w_1$]
       (w1) {};
    \path (45:1cm)   node [player]
     [label=right:$v_4$]
       (v4) {};
    \path (0:1cm)   node [player]
     [label=right:$v_3$]
      (v3) {};
      \path (-45:1cm)   node [player]
     [label=below:$v_2$]
      (v2) {};
       \path (270:1cm)   node [player]
     [label=below:$v_1$]
      (v1) {};
    \draw (270:2cm) node (name) {$D_2$};

\draw[black,thick,-stealth] (u) - + (v4);
\draw[black,thick,-stealth] (u) - + (w2);
\draw[black,thick,-stealth] (u) - + (w4);

\draw[black,thick,-stealth] (v1) - + (u);
\draw[black,thick,-stealth] (v2) - + (u);
\draw[black,thick,-stealth] (v2) - + (v3);
\draw[black,thick,-stealth] (v3) - + (v4);
\draw[black,thick,-stealth] (v4) - + (w1);
\draw[black,thick,-stealth] (w1) - + (w2);
\draw[black,thick,-stealth] (w2) - + (w3);
\draw[black,thick,-stealth] (w3) - + (w4);

\draw[black,thick,-stealth] (w4) - + (x);
\draw[black,thick,-stealth] (v1) - + (x);

    \end{tikzpicture}
\hspace{3em} %2번째 그래프
\begin{tikzpicture}[auto,thick, scale=1.2]
    \tikzstyle{player}=[minimum size=5pt,inner sep=0pt,outer sep=0pt,color=black,fill,draw,circle]
    \tikzstyle{source}=[minimum size=5pt,inner sep=0pt,outer sep=0pt,ball color=black, circle]
    \tikzstyle{arc}=[minimum size=5pt,inner sep=1pt,outer sep=1pt, font=\footnotesize]
  % \draw (22.5:0.35cm) node (name) {$u$};
   \path (0:0cm)   node
   [player]
[label={[xshift=0.38cm, yshift=-0.18cm]:$u$}]
 (u) {};
%Polar coordinate [label={[label distance=1cm]30:label}] {$u$} {};

    \path (225:1cm)   node [player] [label=left:$w_4$] (w4) {};
     \path (247.5:1.5cm)   node [player] [label=left:$x$] (x) {};
    \path (180:1cm)   node [player]
    [label=left:$w_3$]
     (w3) {};
    \path (135:1cm)    node [player]
     [label=left:$w_2$]
     (w2) {};
    \path (90:1cm)   node [player]
    [label=above:$w_1$]
       (w1) {};
    \path (45:1cm)   node [player]
     [label=right:$v_4$]
       (v4) {};
    \path (0:1cm)   node [player]
     [label=right:$v_3$]
      (v3) {};
      \path (-45:1cm)   node [player]
     [label=below:$v_2$]
      (v2) {};
       \path (270:1cm)   node [player]
     [label=below:$v_1$]
      (v1) {};
    \draw (270:2cm) node (name) {$P(D_2)$};

\draw (u) -- (v1);
\draw (u) -- (v2);
\draw (u) -- (v3);
\draw (u) -- (v4);
\draw (u) -- (w1);
\draw (u) -- (w2);
\draw (u) -- (w3);
\draw (u) -- (w4);

\draw (v1) -- (v2);
\draw (v2) -- (v3);
\draw (v3) -- (v4);
\draw (v4) -- (w1);
\draw (w1) -- (w2);
\draw (w2) -- (w3);
\draw (w3) -- (w4);
\draw (w4) -- (v1);
\draw (w4) -- (x);
\draw (v1) -- (x);

    \end{tikzpicture}
\end{center}
\caption{The $(2,3)$ digraphs $D_1$ and $D_2$ and the phylogeny graphs $P(D_1)$ and $P(D_2)$ containing $P_{8} \vee I_1$ and $C_{8} \vee I_1$ as an induced subgraphs, respectively}
\label{fig:phylogeny-$(i,3)$-digraph_fan}
\end{figure}
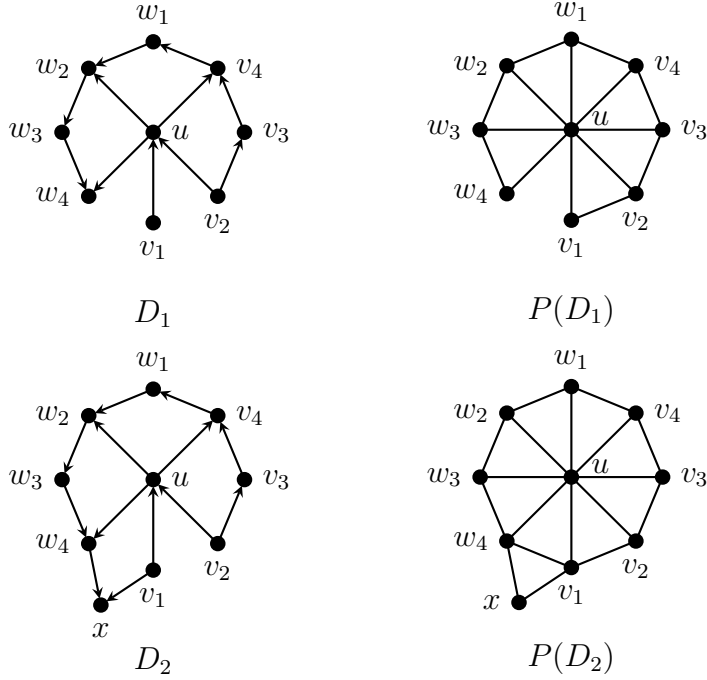

Hereby, we only consider $(i,j)$ phylogeny graphs
for $i \geq 2$ and $j\geq 2$.

The {\it join} of two disjoint graphs $G$ and $H$, denoted by $G \vee H $, is the graph such that $V(G\vee H)=V(G)\cup V(H)$ and $E(G\vee H)=E(G)\cup E(H)\cup \{xy \mid x \in V(G), y \in V(H)\}$.  
\begin{Prop}\label{prop:forbidden_fan}
	If an $(i,j)$ phylogeny graph contains an induced subgraph $H$ isomorphic to a fan $P_\ell \vee I_1$ or a wheel $C_\ell \vee I_1$ for some positive integers $i,j,\ell$ with $i,j \geq 2$, then $\ell \le 2j+2$. Furthermore, $H \cong P_{2j+2} \vee I_1 $ and $H \cong C_{2j+2} \vee I_1 $ are $(i,j)$ realizable.
\end{Prop}
\begin{proof}
Suppose there exists an $(i,j)$ digraph $D$ whose phylogeny graph contains an induced subgraph $H$ isomorphic to a fan $P_\ell \vee \{u\}$ or a wheel $C_\ell \vee \{u\}$ in $P(D)$ for some vertex $u$ in $D$ and some positive integer $\ell$.
	Since $j \geq 2$, $2j+2 \geq 4$ and so the statement is immediately true if $\ell \leq 4$.
	Suppose $\ell >4$.
Then $V(H)- \{u\} \subseteq N(u)$ and any $3$ vertices in $V(H)- \{u\}$ do not form a clique in $P(D)$.
Therefore $|V(H)- \{u\}| \leq 2(j+1)$ by Proposition~\ref{prop:bounded-clique}.
Thus $\ell \leq 2(j+1)$.

To show that $H \cong P_{2j+2} \vee I_1 $ and $H \cong C_{2j+2} \vee I_1 $ are $(i,j)$ realizable, let $D_1$ and $D_2$ be
$(i,j)$ digraphs with the vertex sets
\[V(D_1)=\{u,v_1,v_2,v_3,v_4,w_1,\ldots,w_{2j-2}\},\quad V(D_2)=V(D_1) \cup \{x\},\]
the arc sets
\begin{align*}
A(D_1)= &
\{(u,w_k)\mid\text{$k$ is an even integer}\} \cup \{(w_k,w_{k+1}) \mid  \text{$1\leq k < 2j-2$}\}\\
&\cup \{(v_1,u),(v_2,u),(v_2,v_3),(v_3,v_4),(u,v_4),(v_4,w_1) \} 
\end{align*}
and  $ A(D_2)=A(D_1)\cup \{(v_1,x),(w_{2j-2},x)\}$ (see Figure~\ref{fig:phylogeny-$(i,3)$-digraph_fan} for the digraphs $D_1$ and $D_2$ when $i=2$ and $j=3$).
Then one may check that $u$ is adjacent to
$v_1,v_2,v_3,v_4,w_1,\ldots,w_{2j-2}$ in both $P(D_1)$ and $P(D_2)$.
Moreover,
$P:=v_1v_2v_3v_4 w_1w_2\cdots w_{2j-2}$
is an induced path of length $2j+1$ in $P(D_1)$
and
$C:=
v_1v_2v_3v_4 w_1w_2\cdots w_{2j-2} v_1$ is an induced cycle of length $2j+2$ in $P(D_2)$, respectively.
Thus $P \vee \{u\}$ and $C \vee \{u\}$ are the desired induced subgraphs.
Therefore the the statement is true.
\end{proof}
We call a vertex of indegree $0$ in a digraph a {\em source}.
The following lemma is immediate consequence from the definition of phylogeny graphs.
\begin{Lem} \label{lem:maximum-cared-edges}
Let $D$ be an $(i,j)$ digraph.
Suppose that a vertex $u$ has $ij$ neighbors in the phylogeny graph of $D$ and $D'$ is the subdigraph induced by $u$ and these $ij$ neighbors.
If $u$ is a source of $D'$,
then the following are true:
\begin{enumerate} [(1)]
\item $u$ has outdegree $j$ in $D'$;
\item  Each out-neighbor $v$ of $u$ has indegree $i$ in $D'$ and $N^-_D(v) \cap N^+_D(u) =\emptyset$;
\item $N^-_D(v)\cap N^-_D(w)  =\{u\}$ for each pair $\{v,w\}$ of the out-neighbors of $u$.
\end{enumerate}
\end{Lem}

\begin{Lem} \label{lem:existence_source}
Let $D$ be an acyclic digraph.
If $u$ is a source in $D$ and each of its out-neighbors has indegree at least $2$, then $D$ has a source distinct from it.
\end{Lem}
\begin{proof}
Suppose that $u$ is a source in $D$ and each of its out-neighbors has indegree at least $2$.
Then, since $D$ is acyclic, $D-u$ is acyclic and so there exists a source $v$ in $D-u$.
Since each out-neighbor of $u$ has indegree at least $2$ in $D$, $v$ is not an out-neighbor of $u$ in $D$.
Therefore $v$ is a source in $D$.
\end{proof}

Now we are ready to extend Theorem~\ref{thm:(2,2)-K_5-free} in Lee~{\em et al.}~\cite{lee2017phylogeny} for an $(i,j)$ digraph.

\begin{Thm}\label{thm:complete-free}
Let $G$ be an $(i,j)$ phylogeny graph for positive integers $i,j$ with $i,j \geq 2$.
Then $\omega(G) \leq ij$ and the inequality is tight for $i\leq 3$ and $j=2$.
\end{Thm}

\begin{proof}
To reach a contradiction, suppose that there is an $(i,j)$ digraph $D$ whose phylogeny graph $P(D)$ contains an induced subgraph $H$ isomorphic to $K_{ij+1}$.
Let $D_1$ be the subdigraph of $D$ induced by $V(H)$.
Since $D_1$ is acyclic, there is a source $u$ in $D_1$.
Then $|N^+_{D_1}(u)|=j$ by Lemma~\ref{lem:maximum-cared-edges}(1). 
Therefore \[N^+_{D_1}(u)=N^+_{D}(u).\]
Let \[\langle v \rangle^-=N^-_{D_1}(v)- \{u\}\]
 for each $v \in N^+_{D}(u)$.
 Then, since $ij$ edges in $H$ are incident to $u$, for each $v \in N^+_D(u)$,
\begin{equation}\label{eq:thm:k_ij+1-new0_1}
|\langle v \rangle^-|=i-1, \quad  \quad \langle v \rangle^- \subseteq V(H),
\end{equation}
and
\begin{equation}\label{eq:thm:k_ij+1-new0_2}
\langle v \rangle^-\cap N^+_D(u)=\emptyset
\end{equation}
by
Lemma~\ref{lem:maximum-cared-edges}(2) and
\begin{equation}\label{eq:thm:k_ij+1-new1}
\langle v \rangle^- \cap \langle w \rangle^-  = \emptyset
 \end{equation}
 for each pair $\{v,w\}$ of the out-neighbors of $u$ by
Lemma~\ref{lem:maximum-cared-edges}(3).
Therefore $N^-_{D_1}(v)=
N^-_{D}(v)$ by \eqref{eq:thm:k_ij+1-new0_1}
and so
\[\langle v \rangle^-=N^-_D(v)- \{u\}\] for each $v \in N^+_{D}(u)$.
We note that $|N^+_{D}(u)|=j$ and $|\langle v \rangle^-|=i-1$ for each $v \in N^+_D(u)$ by \eqref{eq:thm:k_ij+1-new0_1}.
Therefore
\begin{equation}\label{eq:thm:k_ij+1-disjoint}
V(H)= \{u\} \sqcup N^+_D(u)
\sqcup
\left(\bigsqcup_{v \in N^+_D(u)} \langle v \rangle^- \right).
\end{equation}
Since each out-neighbor of $u$ has indegree $i$ in $D_1$ by Lemma~\ref{lem:maximum-cared-edges}(2),
$D_1$ has a source $w$ distinct from $u$ by Lemma~\ref{lem:existence_source}.
 Then $w \in \bigsqcup_{v \in N^+_D(u)} \langle v \rangle^- $ by \eqref{eq:thm:k_ij+1-disjoint} and so $w \in \langle v_1\rangle^-$ for some $v_1 \in N^+_D(u)$. 
 Thus, by \eqref{eq:thm:k_ij+1-new1}, \begin{equation}\label{eq:thm:k_ij+1-new2}
w \not \to v
\end{equation}
for any $v \in N^+_D(u) - \{v_1\}$.
Since $w$ is a source, by Lemma~\ref{lem:maximum-cared-edges}(1) and (2) applied to $w$,
the out-neighbors of $w$ belong to $V(H)$ and
\begin{equation}\label{eq:thm:k_ij+1-new3}
w \not \to x
\end{equation} for any $x \in \langle v_1 \rangle^-$.
Take $v_2$ in $N^+_D(u)- \{v_1\}$.
Then, by \eqref{eq:thm:k_ij+1-new2}, $w \not \to v_2$.
Moreover, since $w$ is a source, $v_2 \not \to w$ and so
$w$ and $v_2$ have a common out-neighbor $y_1$ in $V(H)$.
By \eqref{eq:thm:k_ij+1-new3},
$y \notin \langle v_1 \rangle^-$.
Then, since $u \to v_2$, and $v_2 \to y_1$,
$u \neq y_1$.
If $y_1 \in N^+_D(u)$, then $v_2 \in \langle y_1\rangle^-$, which contradicts \eqref{eq:thm:k_ij+1-new0_2}.
Thus $y_1 \notin N^+_D(u)$.
Then, since $y_1 \neq u$, $y_1 \in \bigsqcup_{v \in N^+_D(u)} \langle v \rangle^-$ by \eqref{eq:thm:k_ij+1-disjoint}.
 Therefore $y_1 \in \langle v_3\rangle^-$ for some $v_3 \in N^+_D(u)$.
Then $y_1 \to v_3$.
In addition,  since $w\to y_1$,
$y_1 \notin \langle v_1 \rangle^-$ by \eqref{eq:thm:k_ij+1-new3}
and so 
$v_3 \neq v_1$.
We obtain a directed path $P_1:= v_2 \to y_1 \to v_3$ whose sequence has two vertices $v_2$ and $v_3$ belonging to $N^+_D(u)-\{v_1\}$.
We note that we only used the fact that $v_2$ belongs to $N^+_D(u)- \{v_1\}$ to derive the directed path $P_1$.
Since $v_3$ also belongs to $N^+_D(u)- \{v_1\}$,
we may apply the same argument to obtain a directed path $P_2 :=v_3 \to y_2 \to v_4$ for some vertices $y_2$ in $V(H)$ and $v_4$ in $N^+_D(u) - \{v_1\}$. 
In the above process, we observe that a directed path $P_1 \to P_2$ was obtained where $P_a = v_{a+1} \to y_{a} \to v_{a+2}$  for each $a \in \{1,2\}$,  
$\{v_2,v_3,v_4\} \subseteq N^+_D(u) - \{v_1\}$,  and $\{y_1,y_2\}\subseteq V(H)$.
We continue in this way to obtain the directed walk $P:=P_1 \to P_2 \to \cdots \to P_{j-1}$.
By the way, $P$ contains a closed directed walk.
For, $v_2,\ldots,v_{j+1}$ belong to $N^+_D(u) - \{v_1\}$ and $|N^+_D(u) - \{ v_1 \}|=j-1$ (recall that $u$ has outdegree $j$) and so
$v_l=v_m$ for some distinct integers $l,m \in \{ 2,\ldots,j+1 \}$.
Therefore we reach a contradiction to the fact that $D$ is acyclic.
Hence $P(D)$ is $K_{ij+1}$-free.
Consequently,
if an $(i,j)$ phylogeny graph
contains an induced subgraph isomorphic to $K_{l}$ for positive integers $i\geq 2$, $j \geq 2$ and $l$,
then $l \leq ij$.

By the digraphs given in Figures \ref{fg:2,2digraph} and
\ref{fg:3,2digraph}, the
inequality is tight when $i\leq 3$ and $j=2$.
Therefore the statement is true.
\end{proof}

\begin{Lem} \label{lem:expanding_complete}
If there exists an $(i,j)$ phylogeny graph containing an induced subgraph $H$ isomorphic to $K_l$ for a positive integer $l\geq2$, then, for any positive integer $m$,  there exists an $(i+m,j)$ phylogeny graph containing an induced subgraph isomorphic to $K_{l+m}$.
\end{Lem}

\begin{proof}
Suppose that there exists an $(i,j)$ digraph $D$ whose phylogeny graph $P(D)$ contains an induced subgraph $H$ isomorphic to $K_l$ for some positive integer $l\geq 2$.
To show the statement, it suffices to construct an $(i+1,j)$ digraph whose phylogeny graph contains an induced subgraph isomorphic to $K_{l+1}$.

Let $D_1$ be the subdigraph induced by $V(H)$ of $D$.
Since $D$ is acyclic, $D_1$ is acyclic.
Then there exists a source $u$ in $D_1$.
Take a vertex $v \in V(H)$.
Then $v=u$ or $v \in N^+_{D_1}(u)$ or $u$ and $v$ have a common out-neighbor, i.e. $v\in N^-_D(x)$ for some $x \in N^+_D(u)$.
Since $N^+_{D_1}(u)\subseteq N^+_D(u)$,
we have shown 
\begin{equation} \label{eq:lem:expanding_complete-1}
V(H) \subseteq \left(
\bigcup_{x \in N^+_D(u)} N^-_D(x) \right)  \cup N^+_D(u).
\end{equation}
Then, since $|V(H)|=l \geq 2$,
there exists a vertex $y$ in $V(H)$ distinct from $u$.
Therefore $y \in
\bigcup_{x \in N^+_D(u)} N^-_D(x) $ or $y\in N^+_D(u)$.
Thus, in each case, we show
$N^+_D(u)\neq \emptyset$.
We add a new vertex $w$ and the arc set $\{(w,x) \mid (u,x) \in A(D)\}$ to $D$.
Then the resulting digraph $D'$ is an $(i+1,j)$ digraph.
Moreover, for each $v$ in $V(H)$,
$v$ and $w$ have a common out-neighbor or $w \to v$ by \eqref{eq:lem:expanding_complete-1}.
Thus $V(H) \cup \{w\}$ forms a clique in $P(D')$.
\end{proof}

Given an $(i,2)$ phylogeny graph $G$ for positive integer $i\geq 2$,
$\omega(G) \leq 2i$ by Theorem~\ref{thm:complete-free}.
Further, if $i \geq 4$,
 $\omega(G) \leq \frac{3i}{2}$, which is strictly less than $2i$, by the following theorem.
\begin{Thm}\label{thm:complete-free_j=2}
Let $G$ be an $(i,2)$ phylogeny graph for a positive integer $i\geq 4$.
Then $\omega(G) \leq \frac{3i}{2} +1$ and the inequality is tight.
\end{Thm}
\begin{proof}
Suppose, to the contrary, that there exists an $(i,2)$ digraph $D$ whose phylogeny graph contains an induced subgraph isomorphic to $K_l$ for some positive integers $i \geq 4$ and $l > \frac{3i}{2}+1$.
It suffices to consider the case where $l$ is the minimum satisfying the inequality and so we may assume
\[l = \left\lfloor \frac{3i}{2} \right\rfloor +2= \left\lfloor \frac{i}{2}\right\rfloor+i+2.\]
%display에서 floor 기호 크게
If $i$ is even, then there exists an $(i+1,2)$ digraph $\hat{D}$ whose phylogeny graph contains an induced subgraph isomorphic to $K_{\left\lfloor \frac{(i+1)}{2}\right\rfloor+(i+1)+2}$
by Lemma~\ref{lem:expanding_complete} (note that $\left(\frac{i}{2}+i+2\right)+1=\left\lfloor \frac{(i+1)}{2}\right\rfloor+(i+1)+2$).
% and we can check that $\left(\frac{i}{2}+i+2\right)+1=\left\lfloor \frac{(i+1)}{2}\right\rfloor+(i+1)+2.$
%$
%=K_{\left\lfloor \frac{i+1}{2}\right\rfloor+i+1+2.}$ $K_{{\left(\frac{i}{2}+i+2\right)+1}={\left\lfloor \frac{i+1}{2}\right\rfloor+i+1+2.}}$ 
By replacing $D$ with $\hat{D}$, we may assume that $i$ is odd. 
Therefore $D$ is a $(2k-1,2)$ digraph whose phylogeny graph contains an induced subgraph $H$ isomorphic to $K_{3k}$ for some integer $k\geq 3$.
We assume that 
\begin{enumerate}
\item[(A)]
$D$ has the smallest number of arcs among the $(i,2)$ digraphs with the vertex set $V(D)$ whose phylogeny graphs contain a subgraph isomorphic to $K_{3k}$.
\end{enumerate}
Then 
\begin{enumerate}
\item[(B)] every vertex not belonging to $V(H)$ has no out-neighbor in $D$.
\end{enumerate}
Let $D_1$ be the subdigraph induced by $V(H)$ of $D$. 
Then $V(D_1)$ forms a clique in $P(D)$.
Moreover, $D_1$ is acyclic and so $D_1$ has a source.
If a source in $D_1$ has outdegree at most $1$ in $D$, then it is adjacent to at most $2k-1$ vertices in $H$ and so it
is nonadjacent to some vertex in $H$ in $P(D)$, which is impossible.
Therefore each source in $D_1$ has outdegree $2$ in $D$.
Take a source $u$ in $D_1$. Then
\[
N^+_D(u)=\{v,w\}
\]
for some vertices $v,w$ in $D$.
For simplicity, we let 
%\[ [u]^+=N^+_{D_1}(u) , \quad [v]^-= N^-_{D_1}(v), \quad \text{and} \quad [w]^-=N^-_{D_1}(w)\]
%(note that $[u]^+ $, $[v]^-$, and $[w]^-$ are subsets of $V(H)$). 
\[ [u]^+=N^+_{D_1}(u) , \quad [v]^- =N^-_{D}(v)\cap V(H), \quad \text{and} \quad [w]^-=N^-_{D}(w)\cap V(H).\]
Take a vertex $h \in V(H) -\{u\}$.
Then $h$ is adjacent to $u$.
Therefore one of the following is true:
 (i) $h \in [u]^+$; (ii) $h$ is an in-neighbor of exactly one of $v$ and $w$, that is, $h$ belongs to the symmetric difference $[v]^- \bigtriangleup [w]^-$; (iii) $h$ is a common in-neighbor of $v$ and $w$, i.e. $h \in [v]^-  \cap  [w]^- - \{u\}$.
This together with the fact that $u \in [v]^-  \cap  [w]^-$ implies that
 \begin{equation}\label{eq:thm:ij-general-vertices}	
V(H)= \left([v]^- \bigtriangleup [w]^-  -[u]^+  \right)\sqcup [u]^+ \sqcup \left( [v]^-  \cap  [w]^-\right). \end{equation}
%We call a vertex in $[v]^- \bigtriangleup [w]^-  -[u]^+ $ 
Thus
 \begin{equation}\label{eq:thm:ij-general}	
3k= \left\vert[v]^- \bigtriangleup [w]^-  -[u]^+  \right \vert +  \left\vert[u]^+\right \vert  + \left\vert  [v]^-  \cap  [w]^-\right \vert.
 \end{equation}
Hence
\[  3k \leq  \left\vert[v]^- \bigtriangleup [w]^-  \right \vert +  \left\vert[u]^+\right \vert  + \left\vert  [v]^-  \cap  [w]^-\right \vert = \left\vert  [v]^-\cup    [w]^-\right \vert+ \left\vert[u]^+\right \vert.\] 
Therefore
\begin{align*}
\left\vert[v]^- - [w]^-\right\vert &=\left\vert[v]^-\cup [w]^- \right\vert - \left\vert[w]^-\right\vert
\geq \left\vert[v]^-\cup [w]^-\right\vert - (2k-1)
\\
&\geq
 3k-\left\vert[u]^+\right\vert  - (2k-1) = k+1-\left\vert[u]^+\right\vert
\end{align*}
and so 
\begin{equation}\label{eq:thm:ij-1-twin}
\left\vert[v]^--[w]^-\right\vert  \geq  k+1-\left\vert[u]^+\right\vert .
\end{equation}
Since there is no distinction between $v$ and $w$,
\begin{equation}\label{eq:thm:ij-1}
 \left\vert[w]^--[v]^-\right\vert  \geq k+1-\left\vert[u]^+\right\vert .
\end{equation}
%display아닌 경우에 분수대신에 i/2로 표현
By our assumption, $k\geq 3$ and $|[u]^+|\leq 2$.
Therefore we have the following by \eqref{eq:thm:ij-1-twin} and \eqref{eq:thm:ij-1}:
\begin{equation}\label{eq:thm:ij-1-1}
\left\vert [v]^- - [w]^-\right\vert  \geq 2 \quad \text{and} \quad
\left\vert [w]^- - [v]^-\right\vert  \geq 2.
\end{equation}
We know $u$ has outdegree $2$ in $D$ but we do not know if these out-neighbors are both in $D_1$. Suppose they are not. That is, suppose $|[u]^+|=0$.
Then $\{v,w\} \cap V(H)=\emptyset$ and so, by the property (B), $N^+_D(v)=N^+_D(w)=\emptyset$.
Let $D_2$ be the subdigraph of $D_1$ induced by $[v]^- \bigtriangleup [w]^-$.
Then $D_2$ is acyclic and so $D_2$ has a source, namely $u'$.
Without loss of generality,
we may assume $u' \in [w]^- - [v]^-$.
Then we obtain the digraph $D'$ by deleting the arc $(u',w)$ and adding the arc $(v,w)$ in $D$.
We can check that $D'$ is a $(2k-1,2)$ digraph and the graph induced by $(V(H)-\{u'\}) \cup \{v\}$ of $P(D')$ is isomorphic to $K_{3k}$. 
Thus 
it suffices to consider the case
\[\left\vert[u]^+\right\vert\geq 1.\]

\begin{claim1}
If $v \in [u]^+$ (resp.\ $w \in [u]^+$) and $v \to w$ (resp.\ $w \to v$), then
$N^+_D(v)=\{w\}$ (resp.\ $N^+_D(w)=\{v\}$).
 \end{claim1}
\begin{proof}[Proof of Claim A]
Suppose that $v \in [u]^+$, $v \to w$, and $N^+_D(v)\neq \{w\}$.
Then $v$ has the other out-neighbor $w'$ distinct from $w$.
	By \eqref{eq:thm:ij-general-vertices}, 
$V(H)=[v]^- \cup [w]^- \cup [u]^+$.
Let $D''$ be the digraph obtained from $D$ by deleting the arc $(v,w')$. 
%We consider the digraph $D'':=D-(v,w')$.
The adjacency of two vertices except $v$ in $V(H)$ does not change in $P(D'')$.
Now take $a \in V(H)-\{v\}$.
Then $ a\in [v]^- \cup [w]^- \cup [u]^+$.
If $a \in [v]^-$, then $(a,v)\in A(D'')$.
If $a \in [w]^-$, then $w$ is a common out-neighbor of $a$ and $v$ in $D''$.
If $ a\in [u]^+$, then $a=w$ and so $(v,a)=(v,w)\in A(D'')$.
Thus, in each case, $v$ is adjacent to $a$ in $P(D'')$.
Hence $P(D'')$ contains $H$.
Since $D'' \neq D$, we reach a contradiction to the property (A).
Thus $N^+_D(v)=\{w\}$. 
Since there is no distinction between $v$ and $w$,
$N^+_D(w)=\{v\}$ if $w \in [u]^+$ and $w \to v$.
Therefore Claim A is true.
\end{proof}

Without loss of generality, we may assume
\[ v \in [u]^+.\]
\begin{claim2}
 If there is no arc between $v$ and $w$, then
 there exists a vertex $v'$ distinct from $v$ and $w$ such that
 \[\left([v]^- -[w]^-\right) \sqcup \left([w]^- -[v]^- -\{v'\}\right) \subseteq N^-_D(v').
 \]
 \end{claim2}
 \begin{proof} [Proof of Claim B]
 Suppose that there is no arc between $v$ and $w$.
 There exists a vertex $w_1$ in $[w]^- - [v]^-$ by \eqref{eq:thm:ij-1-1}.
Suppose $w_1 \to a $ for some $a \in [v]^- - [w]^-$.
Then
$N^+(w_1)=\{a,w\}$.
Since $v \in [u]^+\subseteq V(D_1)$, $w_1$ and $v$ are both in $D_1$ and so are adjacent in $P(D)$.
Then, since $v \not \to a$ and there is no arc between $v$ and $w$,
$w_1$ and $v$ have no common out-neighbor and so $v \to w_1$.
Thus $v \to w_1 \to a \to v$ is a directed cycle, which is impossible.
Hence $N^+_D(w_1) \cap \left([v]^- - [w]^-\right) = \emptyset$.
%Thus $w_1 \not \to v_1$.
Since $w_1$ was arbitrarily chosen from $[w]^- - [v]^-$,
we conclude that 
\begin{equation}\label{eq:thm:ij-1-z}
	N^+_D(z) \cap \left([v]^- - [w]^-\right) = \emptyset
\end{equation}
for each $z$ in $[w]^- - [v]^-$.
Take a vertex $v_1$ in $[v]^- - [w]^-$.
If $N^+_D(v_1) = \{v\}$, then $v_1$ is not adjacent to any vertex in $[w]^- - [v]^-$ by \eqref{eq:thm:ij-1-z} (note that $v\not \to w$).
Thus $N^+_D(v_1)=\{v,v'\}$ for some vertex $v'$ distinct from $v$.
Then $v' \neq w$.
Take a vertex $w_1$ in $[w]^- - [v]^- - \{v'\}$.
Then $v_1$ and $w_1$ are adjacent.
By \eqref{eq:thm:ij-1-z}, $w_1 \not \to v_1$.
Since $w_1 \not \to v$, $w_1 \to v'$.
Since $w_1$ was arbitrarily chosen from $[w]^- - [v]^- - \{v'\}$,
$v'$ is an out-neighbor for each vertex in $[w]^- - [v]^- - \{v'\}$.
Thus
\begin{equation}\label{eq:thm:ij-1-z_1}
\{v_1\} \sqcup \left([w]^- - [v]^- - \{v'\} \right)\subseteq N^-_D(v').
\end{equation}
Suppose $v_2 \not \to v'$ for some $v_2$ in $[v]^- - [w]^- - \{v_1\}$.
There exists a vertex $w_2$ in $[w]^- - [v]^- - \{v'\}$ by \eqref{eq:thm:ij-1-1}.
Then $w_2 \to w$. Moreover, by \eqref{eq:thm:ij-1-z_1}, $w_2 \to v'$.  
Thus $N^+(w_2)=\{v',w\}.$
Hence $v_2$ and $w_2$ have no common prey. 
Then, since $v_2$ and $w_2$ are adjacent and $w_2 \not \to v_2$ by \eqref{eq:thm:ij-1-z}, $v_2 \to w_2$ and so $N^+_D(v_2)=\{v,w_2\}$.
If $v' \in [w]^- - [v]^-$, then $v' \not \to v_2$ by \eqref{eq:thm:ij-1-z} and so $v'$ and $v_2$ are not adjacent, which is impossible.
Thus $v' \not \in [w]^- - [v]^-$.
Then there exists a vertex $w_3$ in $[w]^- - [v]^- - \{w_2\}$ by \eqref{eq:thm:ij-1-1}.
Since $v' \not \in [w]^- - [v]^-$, $w_3 \neq v'$.
In addition, $v_2 \not \to w_3$ and $w_3 \not \to v_2$.
Since $v_2$ and $w_3$ are adjacent, $w_3 \to w_2$ and so, by \eqref{eq:thm:ij-1-z_1}, $\{v',w,w_2\}\subseteq N^+_D(w_3)$, which is impossible.
Thus there exists an arc from any vertex in $[v]^- - [w]^- - \{v_1\}$
to $v'$, that is, $[v]^- - [w]^- - \{v_1\} \subseteq N^-_D(v').$
 Hence, by \eqref{eq:thm:ij-1-z_1},
 \[\left([v]^- - [w]^- \right) \sqcup \left([w]^- - [v]^- - \{v'\}\right) \subseteq N^-_D(v').\qedhere \]
\end{proof}

{\it Case 1}. $|[u]^+|=2$.
Then $[u]^+=\{v,w\}$.
If $w \to y$ for some $y \in [v]^-$ and
$v \to z$ for some $z \in [w]^-$, then
$w \to y \to v \to z \to w$ is a closed directed  walk, which is impossible.
Thus either
$w \not \to y$ for any $y \in [v]^-$  
or $v \not \to z$ for any $z \in [w]^-$.
Since $[u]^+=\{v,w\}$, we may assume, without loss of generality, that
\begin{equation}\label{eq:thm:ij-1-w_not_y}
	w \not \to y
\end{equation}
for any $y \in [v]^-$.
There exists a vertex $v_1$ in $[v]^- -[w]^- -\{w\}$ by \eqref{eq:thm:ij-1-1}.
Then $w \not \to v_1$, $v_1 \not \to w$, and $w$ and $v_1$ are adjacent in $P(D)$.
Therefore there exists a common out-neighbor $v^*$ of $w$ and $v_1$.

{\it Subcase 1-1}. $w \to v$. Then, by Claim A, $N^+(w)=\{v\}$ and so $v^*=v$.
Thus
\[
v \in N^+_D(v_1).
\]
Since $w \to v$, \[v \not \to w.\]
If $ v \to z$ for some $z \in [w]^-$,
then $v \to z \to w \to v$ is a directed cycle.
Thus
\begin{equation} \label{eq:thm:ij-1_v_not_to_z}
v \not \to z
\end{equation}
 for any $z \in [w]^-$.
There exists a vertex $w_1$ in $[w]^- -[v]^-$ by \eqref{eq:thm:ij-1-1}.
Then $v \not \to w_1$, $w_1 \not \to v $, and $v$ and $w_1$ are adjacent.
Therefore there exists a common out-neighbor $w^*$ of $v$ and $w_1$.
Thus $w^* \not \in [v]^-$ and $w^* \neq w$. Hence $N^+_D(w_1)=\{w,w^*\}$.
In addition,
$w^* \not\in [w]^-$ by \eqref{eq:thm:ij-1_v_not_to_z} and so
\[ w^* \not \in [v]^- \cup [w]^-.\]
Since $v_1$ and $w_1$ are adjacent and $w_1 \not \to v_1$,
$v_1 \to w^*$ or $v_1 \to w_1$.
By the way, there exists a vertex $w_2$ in $[w]^- -[v]^- -\{w_1\}$ by \eqref{eq:thm:ij-1-1}.
Since $v \not \to w_2$ and $w_2 \not \to v$,
$v$ and $w_2$ have a common out-neighbor $w^{**}$.
Then $w^{**} \notin [w]^-$ by \eqref{eq:thm:ij-1_v_not_to_z}. 
In addition, $w^* \not \in [v]^-$ and $w^{**} \neq w$.
Then $N^+_D(w_2)=\{w,w^{**}\}$.

Suppose $v_1 \to w_1$.
Then $N^+_D(v_1)=\{v,w_1\}$.
Since $v_1$ and $w_2$ are adjacent, $w^{**}=w_1$ and so $w^{**} \in [w]^-$, which is impossible.
Thus $v_1 \to w^*$ and so $N^+_{D}(v_1)=\{v,w^*\}$.
Since $v_1$ and $w_2$ are adjacent and $w^*\neq w_2$, we conclude $w^*=w^{**}$.
Since $w_2$ was arbitrarily chosen from $[w]^- -[v]^- -\{w_1\}$,
\[ [w]^- -[v]^- \subseteq N^-_D(w^*).\]
Moreover, $v_1 \in N^-_D(w^*)$.
Since $v_1$ was arbitrary chosen from $[v]^- - [w]^- -\{w\}$, 
\[ [v]^- -[w]^- - \{w\} \subseteq N^-_D(w^*).\]
Since $w \to v$, 
\begin{equation}\label{eq:thm:ij-1_w_v}
v \notin [w]^- - [v]^-.
\end{equation}
Then, since $v \in N^-_D(w^*)$,
\[ \left( [v]^- -[w]^- - \{w\} \right) \sqcup  \left( [w]^- -[v]^- \right)  \sqcup  \{v\} \subseteq N^-_D(w^*).\]

Let $s= |[v]^- -[w]^- -\{w\}|$ and $t= |[w]^- -[v]^-|$.
Since $w \to v$, \begin{equation}\label{eq:thm:ij-1_w_v2}
 	w \in [v]^- - [w]^-. 
 \end{equation}
Then, by \eqref{eq:thm:ij-1-twin} and \eqref{eq:thm:ij-1},
$s \geq k-2$ and $t \geq k-1$.
Then, since $|N^-_D(w^*)|\leq 2k-1$, \[(s,t) \in \left\{(k-2,k-1),(k-1,k-1),(k-2,k)\right\}. \]
and so $s+t \leq 2k-2$.
Then, since $[u]^+=\{v,w\}$, \[[v]^- \triangle [w]^- - [u]^+= ([v]^- - [w]^- - \{w\}) \sqcup ([w]^- - [v]^-)\] by \eqref{eq:thm:ij-1_w_v} and \eqref{eq:thm:ij-1_w_v2}, and so $|[v]^- \triangle [w]^- - [u]^+|=s+t$.
Accordingly, by \eqref{eq:thm:ij-general},
\begin{equation}\label{eq:thm:ij-1-v1_new}
	\left\vert[v]^- \cap [w]^-\right\vert=3k-(s+t+2) \geq 3k-2k=k.
\end{equation}
Thus
$|[v]^- \cap [w]^-| \geq k$.
Hence
\[2k-1\geq \left\vert N^-_D(w)\right\vert \geq \left\vert[w]^- -[v]^-\right\vert + \left\vert[v]^- \cap [w]^-\right\vert \geq k-1+ k =2k-1\]
and so $|N^-_D(w)|=2k-1$.
Therefore $t=|[w]^- -[v]^-|=k-1$ and $|[v]^- \cap [w]^-|=k$.
Then $s+t=2k-2$ by~\eqref{eq:thm:ij-1-v1_new} and so $s=|[v]^- -[w]^- -\{w\}|=k-1$.
Further, since $\{w\} \sqcup([v]^- -[w]^- -\{w\}) \sqcup ([v]^- \cap [w]^-)  \subseteq N^-_D(v)$, 
\[  \left\vert\{w\} \right\vert+\left\vert[v]^- -[w]^- -\{w\}\right\vert +\left\vert[v]^- \cap [w]^-\right\vert=1+(k-1)+k=2k  \leq  \left\vert N^-_D(v) \right\vert, \]
which is impossible. Therefore Subcase 1-1 cannot happen, i.e. $w \not \to v$. 

{\it Subcase 1-2}. $v \to w$. Then, by Claim A, $N^+_D(v)=\{w\}$.
Thus $v \not \to z$ for any $z \in [w]^-$.
Hence, by \eqref{eq:thm:ij-1-w_not_y}, the argument obtained by replacing $v$ with $w$ and adjusting other vertices based upon the replacement in the argument for Subcase 1-1 may be applied to reach a contradiction.

{\it Subcase 1-3}. There is no arc between $v$ and $w$.
Then, by Claim B, there exists a vertex $v'$ such that
\[\left([v]^- - [w]^- \right) \sqcup \left( [w]^- - [v]^-  - \{v'\} \right)\subseteq N^-_D(v').\]
Then $v_1 \to v'$.
Since $N^+_D(v_1)=\{v,v'\}$, 
$v'=v^*$.
Since $w \to v^*$, $v^* \not \in [w]^- - [v]^-$
and so
\[\left([v]^- - [w]^- \right) \sqcup \left( [w]^- - [v]^- \right)  \sqcup \{w \} \subseteq N^-_D(v^*).\]
Then, since $|N^-_D(v^*)|\leq 2k-1$, $|[v]^- - [w]^-|=| [w]^- - [v]^-|=k-1$ by \eqref{eq:thm:ij-1-twin} and \eqref{eq:thm:ij-1}.
Hence 
\[\left([v]^- - [w]^- \right) \sqcup \left( [w]^- - [v]^- \right)  \sqcup \{w\} = N^-_D(v^*).\]
Then $v \not \to v^*$.
Take a vertex $z$ in $[w]^- - [v]^-$.
Therefore $N^+_D(z)=\{v^*,w\}$.
Then, since $v \not \to w$, $v$ and $z$ have no common out-neighbor.
Since $v$ is adjacent to $z$,
$z \to v$ or $v \to z$.
Thus $z \in N^+_D(v)$.
Since $z$ was arbitrarily chosen from $[w]^- - [v]^-$,
$[w]^- - [v]^- \subseteq N^+_D(v)$ and so, by \eqref{eq:thm:ij-1-1}, $N^+_D(v)=[w]^- - [v]^-$.
Thus $v$ and $w$ have no common out-neighbor.
Moreover, since $w \not \to v$ and $v \not \to w$,
$v$ and $w$ are not adjacent, which is a contradiction.

{\it Case 2}. $|[u]^+|\neq 2$.
Then, by the above assumption,
\[[u]^+=\{v\}.\]
Thus $w \not \in V(H)$ and so, by the property (B), \[N^+_D(w)=\emptyset.\]
Suppose $v \not \to w$.
Then, by Claim B, there exists a vertex $v'$ such that
\[
 \left([v]^- -[w]^-\right) \sqcup \left([w]^- -[v]^- -\{v'\}\right) \subseteq N^-_D(v').
\]
Then
\[  k + k-1\leq \left\vert([v]^- -[w]^-) \right\vert + \left\vert \left([w]^- -[v]^- -\{v'\}\right) \right\vert \leq \left\vert N^-_D(v')\right\vert \leq 2k-1\]
by \eqref{eq:thm:ij-1-twin} and \eqref{eq:thm:ij-1}.
Thus 
$|[v]^- -[w]^-|=k$, $|[w]^- -[v]^- -\{v'\}|=k-1$,
\begin{equation}\label{eq:thm:ij-1-v1_new-1}
v' \in [w]^- - [v]^-,
\end{equation}
and
\begin{equation}\label{eq:thm:ij-1-v1_new-2}
 \left( [v]^- -[w]^-\right) \sqcup \left([w]^- -[v]^- -\{v'\}\right) = N^-_D(v').
\end{equation}
Then, since $v \not \in [w]^-$, $v \not \in N^-_D(v')$ by \eqref{eq:thm:ij-1-v1_new-2} and so $v \not \to v'$.
Since $|[w]^- - [v]^- -\{v'\}| = k-1\geq 2$,
there are two vertices $w_1$ and $w_2$ in $[w]^- - [v]^- - \{v'\}$.
Then $N^+_D(w_1)=N^+_D(w_2)=\{w,v'\}$ and so each of $w_1$ and $w_2$ shares no out-neighbor with $v$.
Therefore $N^+_D(v)=\{w_1,w_2\}$.
Then, since $\{w_1,w_2 \} \subseteq [w]^- -[v]^- -\{v'\}$,
$v$ and $v'$ have no common out-neighbor by \eqref{eq:thm:ij-1-v1_new-2}.
In addition, since $v' \in [w]^--[v]^-$ by \eqref{eq:thm:ij-1-v1_new-1}, $v' \not \to v$.
Hence $v$ and $v'$ are not adjacent, which is impossible.
Consequently, we have shown 
\[ v \to w.\]
Then, by Claim A, \[N^+_D(v)=\{w\}.\]

Let $D_3$ be the subdigraph of $D$ induced by $[v]^- \bigtriangleup [w]^- - \{v\}$. 
Since $D_3$ is acyclic, $D_3$ has a source, say $x$.
Then $x \in [v]^- - [w]^-$ or $x \in [w]^- - [v]^-$.

Then we claim the following
%12/4 claimC 안에 equation 사용 유무 확인하기 
\begin{claim3}
$x\in [v]^- - [w]^-$ and there exists an out-neighbor $x^*$ of $x$ such that
$[v]^- \bigtriangleup [w]^- -\{v,x^*\} = N^-_D(x^*)$.
\end{claim3}
\begin{proof}[Proof of Claim C]

Let $\alpha$ denote a vertex in $\{v,w\}$ with $x \in [\alpha]^- - [\beta]^-$ where $\beta$ denotes the vertex in $\{v,w\}-\{\alpha\}$.
If $N^+_D(x)=\{\alpha\}$, then $x$ cannot be adjacent to any vertex in $[\beta]^- - [\alpha]^- -\{\alpha\}$ since $x$ is a source in $D_3$.
Thus $N^+_D(x) \neq \{\alpha\}$ and so $N^+_D(x)=\{\alpha,x^*\}$ for some $x^*$. 

To show $N^-_D(x^*) \subseteq   [v]^- \bigtriangleup [w]^- -\{v,x^*\}$, take $b \in N^-_D(x^*)$. 
Then $b \in V(H)$ by the property (B). 
Moreover, $b \neq x^*$.
If $b \in [v]^- \cap [w]^-$, then $N^+_D(b)=\{v,w,x^*\}$, which is impossible.
Therefore $b \not \in [v]^- \cap [w]^-$.
By \eqref{eq:thm:ij-general-vertices},
$b\in [v]^- \bigtriangleup [w]^- -[u]^+$ or $b\in [u]^+$.
If $b \in [u]^+$, then $b=v$ and so $N^+_D(b)=N^+_D(v)=\{w\}$, which implies $b \not \in N^-_D(x^*)$.
Thus
$b\in [v]^- \bigtriangleup [w]^- -[u]^+$.
Then, since $b \neq x^*$, $b\in [v]^- \bigtriangleup [w]^- -\{v,x^*\}$ and so \begin{equation}\label{eq:thm:ij-beta_new3}
N^-_D(x^*) \subseteq   [v]^- \bigtriangleup [w]^- -\{v,x^*\}.
\end{equation}

Since $x \in [\alpha]^- - [\beta]^-$, $x^* \neq \beta$.
Moreover, since $x$ is a source in $D_3$ and $x$ is adjacent to any vertex in $[\beta]^- - [\alpha]^- -\{\alpha,x^*\}$,
$x^*$ is an out-neighbor of any vertex in $[\beta]^- - [\alpha]^- -\{\alpha,x^*\}$.
Thus
\begin{equation} \label{eq:thm:ij-beta_new}
 [\beta]^- - [\alpha]^- -\{\alpha,x^*\}  \subseteq N^-_D(x^*)
 \end{equation}
and 
\begin{equation} \label{eq:thm:ij-beta}
N^+_D(z)=\{x^*,\beta \}	
\end{equation}
for each $z$ in $[\beta]^- - [\alpha]^- -\{\alpha,x^*\}$.

To show $[\alpha]^- - [\beta]^- -\{\alpha,x^*\}\subseteq N^-_D(x^*)$,
we note that $[\alpha]^- - [\beta]^- -\{\alpha,x^*\}
=[\alpha]^- - [\beta]^- -\{x^*\}$.
Take a vertex $y_1$ in $[\alpha]^- - [\beta]^- -\{x^*\}$.
To the contrary, suppose $y_1\notin N^-_D(x^*)$, i.e. $y_1 \not \to x^*$.
Then, since $x \to x^*$, $y_1\neq x$.
Take $z_1$ in $[\beta]^- - [\alpha]^- -\{\alpha,x^*\}$.
Such a vertex exists since $|[\beta]^- - [\alpha]^- -\{\alpha,x^*\}| \geq 1$ by \eqref{eq:thm:ij-1-twin} and \eqref{eq:thm:ij-1}.
Then $y_1$ and $z_1$ are adjacent, so
$y_1 \to z_1$ by \eqref{eq:thm:ij-beta}.
Therefore $N^+(y_1)=\{\alpha,z_1\}$.
Since $z_1$ was arbitrarily chosen from $[\beta]^- - [\alpha]^- -\{\alpha,x^*\}$, $[\beta]^- - [\alpha]^- -\{\alpha,x^*\}\subseteq N^+(y_1)$ and so $[\beta]^- - [\alpha]^- -\{\alpha,x^*\}=\{z_1\}$.
Thus $[\beta]^- - [\alpha]^- \subseteq \{\alpha,x^*,z_1\}$.
Since $|[\beta]^- - [\alpha]^-|\geq k \geq 3$ by \eqref{eq:thm:ij-1-twin} and \eqref{eq:thm:ij-1},
$[\beta]^- - [\alpha]^- =\{\alpha,x^*,z_1\}$.
Then $x^*$ and $y_1$ are adjacent.
Since $y_1 \not \to x^*$ and $x^*$ and $y_1$ share no out-neighbor,
$x^* \to y_1$.
Thus $y_1 \to z_1 \to x^* \to y_1$ is a directed cycles, which is impossible.
Hence $y_1 \to x^*$.
Since $y_1$ was arbitrarily chosen from $[\alpha]^- - [\beta]^- -\{x^*\}$,
\[
[\alpha]^- - [\beta]^- -\{\alpha,x^*\} =[\alpha]^- - [\beta]^- -\{x^*\}  \subseteq N^-_D(x^*).
\]
Then
\[
[v]^- \bigtriangleup [w]^- -\{\alpha,x^*\} \subseteq N^-_D(x^*)\] by \eqref{eq:thm:ij-beta_new}.
Since $v \in [v]^- \bigtriangleup[w]^-$, $\alpha=v$  and $[v]^- \bigtriangleup [w]^- -\{v,x^*\} = N^-_D(x^*)$ by \eqref{eq:thm:ij-beta_new3}.
\end{proof}

Since $v\to w$ and $\{v\}=[u]^+ \subseteq [v]^- \bigtriangleup [w]^-$, 
\[\left([v]^- \bigtriangleup [w]^-\right)\sqcup \left([v]^- \cap [w]^-\right)= V(H)\] by \eqref{eq:thm:ij-general-vertices}.
Thus, by Claim C,
\begin{equation} \label{eq:thm:ij-x_last}	
N^-_D(x^*) \sqcup \left(([v]^- \bigtriangleup [w]^-) \cap \{v,x^*\}\right) \sqcup \left([v]^- \cap [w]^-\right)= V(H).
\end{equation}
Then, since $|N^-_D(x^*)|\leq 2k-1$ and $|V(H)|=3k$,
we conclude \begin{equation} \label{eq:thm:ij-x_last_1}
\left\vert[v]^- \cap [w]^- \right\vert \geq k-1.
\end{equation}

By the way, since $[u]^+=\{v\}$ and $v \in [w]^-$,
$|[v]^- \cup [w]^-|=3k$ by \eqref{eq:thm:ij-general-vertices}.
Then, since $|[u]^+|=1$, 
\begin{equation}\label{eq:thm:ij-kk}
\left\vert[v]^- - [w]^-\right\vert \geq k\quad \text{and} \quad\left\vert [w]^- - [v]^-\right\vert \geq k
\end{equation}
by \eqref{eq:thm:ij-1-twin} and \eqref{eq:thm:ij-1}.
Therefore
 $|[v]^- \bigtriangleup [w]^-|\geq 2k$.
 Since $|[v]^- \cup [w]^-|=3k$,
 \[\left \vert [v]^- \cap [w]^- \right  \vert= \left \vert [v]^- \cup [w]^-  \right \vert -  \left \vert [v]^- \bigtriangleup [w]^-\right  \vert \leq 3k-2k=k.\] 
If $|[v]^- \cap [w]^-|=k$, then, by \eqref{eq:thm:ij-kk}, $|[v]^-|=|[v]^- \cap [w]^-|+|[v]^- - [w]^-|\geq 2k$, which is impossible.
Suppose $|[v]^- \cap [w]^-|=k-1$.
Then $|[v]^- \bigtriangleup [w]^-|=2k+1$ and so  \[\left \vert [v]^- - [w]^-\right\vert =k+1 \quad \text{or} \quad \left\vert[w]^- - [v]^-\right\vert=k+1\] 
Hence $|[v]^-|=|[v]^- \cap [w]^-|+|[v]^- - [w]^-|\geq 2k$ or $|[w]^-|= |[v]^- \cap [w]^-|+|[w]^- - [v]^-| \geq 2k$, which is impossible.
Therefore $|[v]^- \cap [w]^-|\leq k-2$, which contradicts \eqref{eq:thm:ij-x_last_1}.
Thus we have shown that there is no $(2k-1,2)$ digraph $D$ whose phylogeny graph contains an induced subgraph isomorphic to $K_{3k}$ and so we conclude $\omega(G) \leq \frac{3i}{2}+1$.

To show that the inequality is tight,
we present a $(2k,2)$ digraph and a $(2k+1,2)$ digraph each of whose phylogeny graphs contains $K_{3k+1}$ and $K_{3k+2}$ as an induced subgraph, respectively, for any integer $k \geq 2$.
%construct => present 로 // suffices to show 지웠음
Fix an integer $k\geq 2$.
Let $D_1$ be a $(2k,2)$ digraph with
\[V(D_1)=\{u,v,w,x_1,x_2,\ldots,x_{2k-1},y_1\ldots,y_{k-1},z\} \]
and
\begin{align*}
A(D_1)=&\{(u,v),(u,w),(v,w),(w,z)\} \cup \bigcup_{i=1}^{2k-1} \{ (x_i,v) \}\\
 &\cup \bigcup_{i=1}^{k-1} \{(y_i,w),(y_i,z)\} \cup \bigcup_{i=1}^{k-1} \{ (x_i,w)\} \cup \bigcup_{i=k}^{2k-1} \{ (x_i,z) \}
\end{align*}
(see the $(4,2)$ digraph given in Figure \ref{fg:4,2digraph} for an illustration).
In the following, we show that
$V(D_1)- \{z\}$ forms a clique of size $3k+1$ in $P(D_1)$.
We note that
\[ N^+(u)=\{v,w\}, \quad v \in N^+(u) \cap N^+(x_i), \quad \text{and} \quad w \in N^+(u) \cap N^+(y_j) \]
 for each $1\leq i \leq 2k-1$ and $1\leq j \leq k-1$.
Therefore,
$u$ is adjacent to the vertices in $V(D_1)- \{u,z\}$.
We can check that
\[ N^+(v)=\{w\}, \quad v\in N^+(x_i), \quad \text{and} \quad w \in N^+(v) \cap N^+(y_j) \]
for each $1\leq i \leq 2k-1$ and $1\leq j \leq k-1$.
Therefore, $v$ is adjacent to the vertices in $V(D_1)- \{v,z\}$.
Since
\[ w \in N^+(x_i) \cap N^+(y_i) \quad \text{and} \quad z \in N^+(w) \cap N^+(x_j) \]
for each $1\leq i \leq k-1$ and $k\leq j \leq 2k-1$,
$w$ is adjacent to the vertices in $V(D_1)-\{w\}$.
Take $x_i$ for some $i \in \{1,\ldots, 2k-1\}$.
Since $\{x_1,\ldots,x_{2k-1} \} \subseteq N^-(v)$,
$\{x_1,\ldots,x_{2k-1} \}$ forms a clique.
If $1\leq i \leq k-1$, then $x_i \to w$ and so $w$ is a common out-neighbor of $x_i$ and $y_j$ for each $1\leq j \leq k-1$.
If $k\leq i \leq 2k-1$, then $x_i \to z$ and so $z$ is a common out-neighbor of $x_i$ and $y_j$ for each $1\leq j \leq k-1$.
Therefore $x_i$ is adjacent to the vertices in $\{y_1,\ldots,y_{k-1}\}$.
Thus $x_i$ is adjacent to the vertices in $V(D_1)- \{x_i,z\}$.
Since $\{y_1,\ldots,y_{k-1}\} \subseteq N^-(w)$,
$\{y_1,\ldots,y_{k-1}\}$ forms a clique.
Therefore we have shown that $V(D_1)- \{z\}$ forms a clique in $P(D_1)$.
Then, by Lemma~\ref{lem:expanding_complete},
we obtain a $(2k+1,2)$ digraph whose phylogeny graph contains an induced subgraph isomorphic to $K_{3k+2}$.
Hence we have shown that the inequality is tight.
 \end{proof}

\begin{figure}
\begin{center}
  \begin{tikzpicture}[auto,thick, scale=0.8]
    \tikzstyle{player}=[minimum size=5pt,inner sep=0pt,outer sep=0pt,fill,color=black, circle]
    \tikzstyle{player2}=[minimum size=2pt,inner sep=0pt,outer sep=0pt,fill,color=black, circle]
    \tikzstyle{source}=[minimum size=5pt,inner sep=0pt,outer sep=0pt,ball color=black, circle]
    \tikzstyle{arc}=[minimum size=5pt,inner sep=1pt,outer sep=1pt, font=\footnotesize]
    \path (1.5,1.5)   node [player]  (v1) [label=right:] {};
  \path (1.5,0)   node [player]  (v2) [label=right:] {};
  \path (0,1.5)   node [player]  (v3) [label=right:] {};
  \path (0,0)   node [player]  (v4) [label=right:] {};
   \path (2.5,1.3)   node [player]  (x) [label=right:] {};
  \draw[black,thick,-stealth] (v1) to (v2);
  \draw[black,thick,-stealth] (v1) to (v3);
  \draw[black,thick,-stealth] (v4) to (v2);
  \draw[black,thick,-stealth] (v3) to (v4);
  \draw[black,thick,-stealth] (v3) to  [in=110, out=65, distance=0.9cm] (x);
  \draw[black,thick,-stealth] (v2) to  [in=-90, out=0, distance=0.5cm] (x);
    \end{tikzpicture}
\hspace{3em}
  \begin{tikzpicture}[auto,thick,scale=0.8]
    \tikzstyle{player}=[minimum size=5pt,inner sep=0pt,outer sep=0pt,fill,color=black, circle]
    \tikzstyle{player2}=[minimum size=2pt,inner sep=0pt,outer sep=0pt,fill,color=black, circle]
    \tikzstyle{source}=[minimum size=5pt,inner sep=0pt,outer sep=0pt,ball color=black, circle]
    \tikzstyle{arc}=[minimum size=5pt,inner sep=1pt,outer sep=1pt, font=\footnotesize]
    \path (1.5,1.5)   node [player]  (v1) [label=right:] {};
  \path (1.5,0)   node [player]  (v2) [label=right:] {};
  \path (0,1.5)   node [player]  (v3) [label=right:] {};
  \path (0,0)   node [player]  (v4) [label=right:] {};
   \path (2.5,1.3)   node [player]  (x) [label=right:] {};
\draw (v1) -- (v2);
\draw (v1) -- (v4);
\draw (v2) -- (v3);
  \draw (v1) -- (v3);
  \draw (v4) --(v2);
  \draw (v3) -- (v4);
 \draw (v3) to [bend left=60 ]  (x);
%  \draw (v2) --(x);
    \draw (v2) to [bend right=30 ]  (x);
    \end{tikzpicture}
    \end{center}
\caption{A (2,2) digraph and its phylogeny graph}
\label{fg:2,2digraph}
 \end{figure}
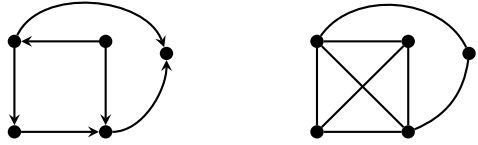

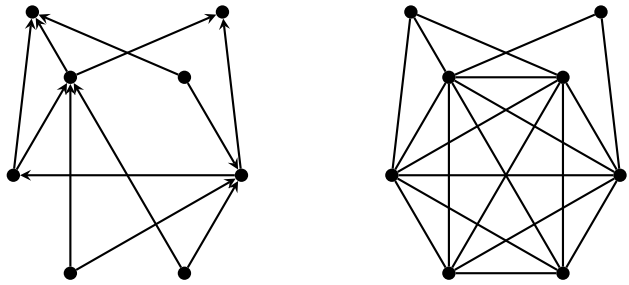
\begin{figure}
\begin{center}
  \begin{tikzpicture}[auto,thick]
    \tikzstyle{player}=[minimum size=5pt,inner sep=0pt,outer sep=0pt,fill,color=black, circle]
    \tikzstyle{player2}=[minimum size=2pt,inner sep=0pt,outer sep=0pt,fill,color=black, circle]
    \tikzstyle{source}=[minimum size=5pt,inner sep=0pt,outer sep=0pt,ball color=black, circle]
    \tikzstyle{arc}=[minimum size=5pt,inner sep=1pt,outer sep=1pt, font=\footnotesize]
    \path (0:1.5cm)    node [player]  (v2) [label=right:] {};
  \path (60:1.5cm)   node [player]  (v1) [label=right:] {};
  \path (120:1.5cm)  node [player]  (v6) [label=right:] {};
  \path (180:1.5cm)   node [player]  (v5) [label=right:] {};
   \path (240:1.5cm)   node [player]  (v4) [label=right:] {};
      \path (300:1.5cm)   node [player]  (v3) [label=right:] {};
      \path (60:2.5cm)   node [player]  (x) [label=right:] {};
      \path (120:2.5cm)   node [player]  (y) [label=right:] {};
  \draw[black,thick,-stealth] (v1) to (v2);
  \draw[black,thick,-stealth] (v1) to (y);
\draw[black,thick,-stealth] (v2) to (x);
\draw[black,thick,-stealth] (v2) to (v5);
\draw[black,thick,-stealth] (v3) to (v2);
\draw[black,thick,-stealth] (v3) to (v6);
\draw[black,thick,-stealth] (v4) to (v2);
\draw[black,thick,-stealth] (v5) to (y);
\draw[black,thick,-stealth] (v5) to (v6);
\draw[black,thick,-stealth] (v4) to (v6);
\draw[black,thick,-stealth] (v6) to (y);
\draw[black,thick,-stealth] (v6) to (x);
    \end{tikzpicture}
\hspace{3em}
   \begin{tikzpicture}[auto,thick]
    \tikzstyle{player}=[minimum size=5pt,inner sep=0pt,outer sep=0pt,fill,color=black, circle]
    \tikzstyle{player2}=[minimum size=2pt,inner sep=0pt,outer sep=0pt,fill,color=black, circle]
    \tikzstyle{source}=[minimum size=5pt,inner sep=0pt,outer sep=0pt,ball color=black, circle]
    \tikzstyle{arc}=[minimum size=5pt,inner sep=1pt,outer sep=1pt, font=\footnotesize]
    \path (0:1.5cm)    node [player]  (v2) [label=right:] {};
  \path (60:1.5cm)   node [player]  (v1) [label=right:] {};
  \path (120:1.5cm)  node [player]  (v6) [label=right:] {};
  \path (180:1.5cm)   node [player]  (v5) [label=right:] {};
   \path (240:1.5cm)   node [player]  (v4) [label=right:] {};
      \path (300:1.5cm)   node [player]  (v3) [label=right:] {};
      \path (60:2.5cm)   node [player]  (x) [label=right:] {};
      \path (120:2.5cm)   node [player]  (y) [label=right:] {};
  \draw (v1) -- (v2);
  \draw (v1) -- (y);
  \draw (v1) -- (v3);
  \draw (v1) -- (v4);
  \draw (v1) -- (v5);
  \draw (v1) -- (v6);
\draw (v2) -- (x);
\draw (v2)-- (v5);
\draw (v2)-- (v6);
\draw (v3) -- (v2);
\draw (v3)-- (v6);
\draw (v3) -- (v4);
\draw (v3)-- (v5);
\draw (v4) -- (v2);
\draw (v4) -- (v5);
\draw (v5) --(y);
\draw (v5) -- (v6);
\draw (v4) -- (v6);
\draw (v6) -- (y);
\draw (v6) -- (x);
    \end{tikzpicture}
    \end{center}
\caption{A (3,2) digraph and its phylogeny graph}
\label{fg:3,2digraph}
 \end{figure}

 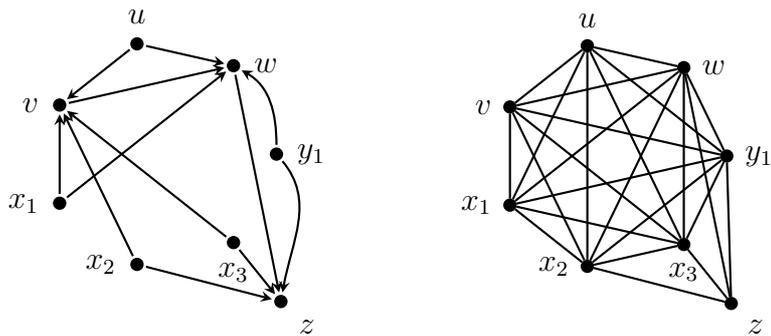
\begin{figure}
\begin{center}
  \begin{tikzpicture}[auto,thick]
    \tikzstyle{player}=[minimum size=5pt,inner sep=0pt,outer sep=0pt,fill,color=black, circle]
    \tikzstyle{player2}=[minimum size=2pt,inner sep=0pt,outer sep=0pt,fill,color=black, circle]
    \tikzstyle{source}=[minimum size=5pt,inner sep=0pt,outer sep=0pt,ball color=black, circle]
    \tikzstyle{arc}=[minimum size=5pt,inner sep=1pt,outer sep=1pt, font=\footnotesize]
  \tikzstyle{player}=[minimum size=5pt,inner sep=0pt,outer sep=1pt,fill,color=black, circle]
    \tikzstyle{player2}=[minimum size=2pt,inner sep=0pt,outer sep=0pt,fill,color=black, circle]
    \tikzstyle{source}=[minimum size=5pt,inner sep=0pt,outer sep=0pt,ball color=black, circle]
    \tikzstyle{arc}=[minimum size=5pt,inner sep=1pt,outer sep=1pt, font=\footnotesize]
    \path (0:1.5cm)    node [player]  (v1) [label=right:$y_1$] {};
  \path (360/7:1.5cm)   node [player]  (v2) [label=right:$w$] {};
  \path (360 * 2/7:1.5cm)  node [player]  (v3) [label=above:$u$] {};
  \path (360 * 3/7:1.5cm)   node [player]  (v4) [label=left:$v$] {};
   \path (360 * 4/7:1.5cm)   node [player]  (v5) [label=left:$x_1$] {};
      \path (360 * 5/7:1.5cm)   node [player]  (v6) [label=left:$x_2$] {};
           \path (360 * 6/7:1.5cm)   node [player]  (v7) [label=below:$x_3$] {};
      \path (360 * 6/7:2.5cm)   node [player]  (x) [label=below right:$z$] {};
  \draw[black,thick,-stealth] (v1) to [in=-30,out=90] (v2);
 \draw[black,thick,-stealth] (v1) to [in=80,out=-50](x);

  \draw[black,thick,-stealth] (v2) to (x);

  \draw[black,thick,-stealth] (v3) to (v2);
  \draw[black,thick,-stealth] (v3) to (v4);

  \draw[black,thick,-stealth] (v4) to (v2);

  \draw[black,thick,-stealth] (v5) to (v4);
\draw[black,thick,-stealth] (v5) to (v2);

\draw[black,thick,-stealth] (v6) to (v4);
\draw[black,thick,-stealth] (v6) to (x);

\draw[black,thick,-stealth] (v7) to (v4);
\draw[black,thick,-stealth] (v7) to (x);
    \end{tikzpicture}
\hspace{3em}
   \begin{tikzpicture}[auto,thick]
    \tikzstyle{player}=[minimum size=5pt,inner sep=0pt,outer sep=0pt,fill,color=black, circle]
    \tikzstyle{player2}=[minimum size=2pt,inner sep=0pt,outer sep=0pt,fill,color=black, circle]
    \tikzstyle{source}=[minimum size=5pt,inner sep=0pt,outer sep=0pt,ball color=black, circle]
    \tikzstyle{arc}=[minimum size=5pt,inner sep=1pt,outer sep=1pt, font=\footnotesize]
    \path (0:1.5cm)    node [player]  (v1) [label=right:$y_1$] {};
  \path (360/7:1.5cm)   node [player]  (v2) [label=right:$w$] {};
  \path (360 * 2/7:1.5cm)  node [player]  (v3) [label=above:$u$] {};
  \path (360 * 3/7:1.5cm)   node [player]  (v4) [label=left:$v$] {};
   \path (360 * 4/7:1.5cm)   node [player]  (v5) [label=left:$x_1$] {};
      \path (360 * 5/7:1.5cm)   node [player]  (v6) [label=left:$x_2$] {};
           \path (360 * 6/7:1.5cm)   node [player]  (v7) [label=below:$x_3$] {};
      \path (360 * 6/7:2.5cm)   node [player]  (x) [label=below right:$z$] {};
  \draw (v1) -- (v2);
  \draw (v1) -- (v3);
  \draw (v1) -- (v4);
  \draw (v1) -- (v5);
  \draw (v1) -- (v6);
  \draw (v1) -- (v7);
  \draw (v1) -- (x);
\draw (v2) -- (x);
\draw (v2)-- (v5);
\draw (v2)-- (v6);
\draw (v3) -- (v2);
\draw (v3)-- (v6);
\draw (v3) -- (v4);
\draw (v3)-- (v5);
\draw (v4) -- (v2);
\draw (v4) -- (v5);

\draw (v5) -- (v6);
\draw (v4) -- (v6);
\draw (v6) -- (x);

\draw (v7) -- (v2);
\draw (v7) -- (v3);
\draw (v7) -- (v4);
\draw (v7) -- (v5);
\draw (v7) -- (v6);
\draw (v7) -- (x);
    \end{tikzpicture}
    \end{center}
\caption{A (4,2) digraph and its phylogeny graph}
\label{fg:4,2digraph}
 \end{figure}

\begin{proof}[Proof of Theorem~\ref{thm:final_forbidden}]
For $i,j\geq 2$,
Propositions~\ref{prop:forbidden_star}, \ref{prop:forbidden_bipartite}, \ref{prop:forbidden_fan}, and Theorems~\ref{thm:complete-free} and \ref{thm:complete-free_j=2} 
provide the given forbidden subgraphs of an $(i,j)$ phylogeny graph.\end{proof}

\section{Closing remarks} \label{sec:number-of}
Given a hole $H$ in the underlying graph of an $(i,j)$ digraph $D$,
it is interesting to ask if the subgraph of $P(D)$ induced by $V(H)$ still contains a hole.
Theorem~\ref{thm:main} partially answers the question for $(i,2)$ phylogeny graphs. 
The authors would like to see the above question answered. 
\section{Acknowledgement}
We are thankful to the reviewer whose suggestions highly advanced the clarity of the paper.

This work was supported by Science Research Center Program through the National Research Foundation of Korea(NRF) grant funded by the Korean Government (MSIT)(NRF-2022R1A2C1009648 and 2016R1A5A1008055).

%\bibliography{12-step}
%%\bibliographystyle{unsrt}
%\bibliographystyle{plain}
%%\bibliography{12-step}

\end{document}